\documentclass[dvips,preprint]{imsart}

\RequirePackage[OT1]{fontenc}
\RequirePackage{amsthm,amsmath,natbib}
\RequirePackage[colorlinks]{hyperref}
\usepackage{amsmath,amssymb,amsthm}
\usepackage[latin1]{inputenc}
\usepackage{mathrsfs}
\usepackage[english]{babel}

\startlocaldefs 
\numberwithin{equation}{section}
\theoremstyle{plain}
\newtheorem{theorem}{Theorem}[section]
\newtheorem{lemma}{Lemma}[section]

\def\@journal{Submitted}      
\endlocaldefs
\numberwithin{equation}{section}  

\usepackage{a4wide}
\usepackage[dvips]{graphicx}
\allowdisplaybreaks

\begin{document}

\begin{frontmatter}
  \title{Corrections to LRT on Large Dimensional Covariance Matrix by
    RMT\protect\thanksref{T2}}
  \runtitle{Corrected LRT on Large covariance matrices}
  \thankstext{T1}{The research of this author was supported by CNSF grant 10571020 and NUS
    grant R-155-000-061-112}

  \thankstext{T2}{This version contains full proofs of all results. A
    shorter  version is to be published in a journal.}

  \begin{aug}
    \author{\fnms{Zhidong} \snm{Bai}\thanksref{T1}\ead[label=e1]{stabaizd@nus.edu.sg}}
    \and
    \author{\fnms{Dandan} \snm{Jiang}\ead[label=e2]{stajd@nus.edu.sg}}
    \and
    \author{\fnms{Jian-feng} \snm{Yao}\ead[label=e3]{jian-feng.yao@univ-rennes1.fr}}
    \and
    \author{\fnms{Shurong} \snm{Zheng}\ead[label=e4]{zhengsr@nenu.edu.cn}}

    \runauthor{Z. Bai, D. Jiang, J.  Yao and S. Zheng}

    \affiliation{Northeast  Normal University,
      National University of Singapore,   IRMAR and Universit\'e de Rennes 1}
    \address{Zhidong Bai, Dandan Jiang and  Shurong Zheng \\
      KLASMOE and School of Mathematics and Statistics\\
      Northeast Normal University \\
      5268 People's Road\\
      130024 Changchun, China\\
      and\\
      Department of Statistics and Applied Probability\\
      National University of Singapore\\
      10, Kent Ridge Crescent\\
      Singapore 119260\\
      \printead{e1}, \printead{e2}, \printead{e4}
    }

    \address{Jian-feng Yao \\
      IRMAR and Universit\'e de Rennes 1\\
      Campus de Beaulieu\\
      35042   Rennes Cedex, France\\
      \printead{e3}
    }
  \end{aug}

  \begin{abstract}
    In this paper, we give an explanation to the failure of two
    likelihood ratio procedures for testing about covariance  matrices
    from Gaussian populations
    when the dimension is large compared to the sample size. Next,
    using recent central limit theorems for linear spectral statistics
    of sample covariance matrices and of random F-matrices,
    we propose necessary corrections for  these LR   tests    to cope with
    high-dimensional effects.
    The asymptotic distributions of these  corrected tests under the
    null are given.  Simulations demonstrate that the corrected LR
    tests   yield a realized size close to nominal level
    for both moderate $p$ (around 20) and high dimension, while the traditional LR tests  with
    $\chi^2$ approximation fails.

    Another contribution from the paper is that for testing the
    equality between  two covariance matrices,  the proposed
    correction  applies  equally for non-Gaussian populations  yielding a
    valid  pseudo-likelihood ratio  test.
  \end{abstract}


\end{frontmatter}

\section{Introduction}
\label{sec:intro}

The rapid development and wide application of computer techniques
permits to collect and store a huge amount data, where the number of
measured variables is usually large. Such high dimensional data
occur in many modern scientific fields, such as micro-array data in
biology, stock market analysis in finance and  wireless
communication networks. Traditional  estimation or test tools are no
more valid, or perform  badly for such high-dimensional data, since
they typically assume a large sample size $n$ with respect to the
number of variables $p$.  A better approach in this high-dimensional
data setting would be based on asymptotic theory which has both $n$
and $p$ approaching infinity. To illustrate this purpose, let us
mention the case of Hotelling's $T^2$-test. The failure of
$T^2$-test for high-dimensional data has been mentioned as early as
by \citet{D58}.  As a remedy, Dempster proposed a so-called
non-exact test. However, the theoretical justification of Dempster's
test arises  much later in \cite{B96} inspired by  modern random
matrix theory (RMT). These authors have found necessary correction
for the $T^2$-test to compensate effects due to high dimension.

In this paper,  we consider two LR tests  concerning covariance
matrices.  We first give a theoretical  explanation for the fail of
these tests in   high-dimensional data context.
Next,  with the aid of random matrix
theory, we provide  necessary corrections to these LR tests to cope
with the   high dimensional effects.

First, we consider the problem of one-sample covariance hypothesis
test. Suppose that $\mathbf x$ follows a $p$-dimensional Gaussian
distribution $N(\mu_p, \Sigma_p)$  and we want to test
\begin{equation}
  H_0 : ~\Sigma_p  = I_p ~,  \label{H0piao}
\end{equation}
where $I_p$ denotes the  $p$-dimensional identity matrix. Note that
testing  $ \Sigma_p= A $  with an arbitrary covariance matrix $A$
can always be reduced to the above null hypothesis by  the
transformation $A^{-\frac{1}{2}}\mathbf{x}$.

Let $(\textbf{x}_1,\cdots,\textbf{x}_n)$ be a
sample from $\mathbf x$, where  we assume  $ p < n $.  The  sample
covariance matrix is
\begin{equation}
  \textbf{S}=\frac{1}{n}\sum\limits_{i=1}^p(\textbf{x}_i-\overline{\textbf{x}})
  (\textbf{x}_i-\overline{\textbf{x}})^*,\label{S}
\end{equation}
and set
\begin{equation}L^*=tr \textbf{S} - \log|\textbf{S}|-p~.   \label{L*}
\end{equation}
The likelihood ratio test statistic  is
\begin{equation}
  T_n=n\cdot L^* .
  \label{singlVclassict}
\end{equation}
Keeping $p$ fixed while letting $n \rightarrow \infty$, then the
classical theory depicts that $T_n$ converges to the
$\chi^2_{\frac{1}{2}p(p+1)}$ distribution
under $H_0$.

However, as it will be shown, this classical  approximation leads to
a test size much higher than the nominal test level in the case of
high-dimensional data,  because $T_n$ approaches infinity for large
$p$.  As seen from Table 1 in \S\ref{sec:test1},  for dimension and
sample sizes $(p, n)= (50, 500)$, the realized size of the test  is
22.5\% instead of the nominal 5\%  level. The result is even worse
for the case $(p, n)= (300, 500)$, with   a 100\% test size.

Based on a recent  CLT for linear spectral statistics (LSS) of
large-dimensional sample covariance matrices \citep{B04}, we
construct a corrected version of $T_n$ in \S\ref{sec:test1}. As
shown by  the simulation results of  \S\ref{sec:simul1}, the
corrected test performs much better in  case of  high dimensions.
Moreover, it also performs correctly   for moderate dimensions like
$p=10$ or  20. For dimension and sample sizes $(p, n)$  cited above,
the sizes of the  corrected test are 5.9\% and 5.2\%, respectively,
both close to the  5\% nominal  level.

The second test problem we consider is about the  equality between
two high-dimensional covariance matrices.
Let ${\bf x}_i = (x_{1i}, x_{2i},\cdots , x_{pi})^{T},
i=1,\cdots,n_1$ and ${\bf y}_j = (y_{1j}, y_{2j}, \cdots,
y_{pj})^{T},$ $ j=1, \cdots, n_2$ be observations from two
$p$-dimensional normal populations $N(\mu_k,\Sigma_k),~ k=1,2$,
respectively. We wish to test the null  hypothesis
\begin{equation}
  H_0: \Sigma_1=\Sigma_2  ~.
  \label{a0}
\end{equation}
The related sample covariance matrices are
$$
\displaystyle{A=\frac{1}{n_1}\sum\limits_{i=1}^{n_1}({\bf
    x}_{i}-\overline{{\bf x}})({\bf x}_{i}-\overline{{\bf x}})^{*},}
\quad \displaystyle{B=\frac{1}{n_2}\sum\limits_{i=1}^{n_2}({\bf
    y}_{i}-\overline{{\bf y}})({\bf y}_{i}-\overline{{\bf y}})^{*},}
$$
where $\overline{{\bf x}}$ , $\overline{{\bf y}}$ are the respective
sample means. Let
\begin{equation}
  L_1=\frac{\left|A\right|^{\frac{n_1}{2}}\cdot
    \left|B\right|^{\frac{n_2}{2}}}
  {\left|c_1A+c_2B\right|^{\frac{N}{2}}},\label{L1AB}
\end{equation}
where  $N=n_1+n_2$  and  $c_k$ denote $\frac{n_k}{N}, k=1,2.$ The
likelihood ratio test statistic  is
$$
T_N=-2\log L_1,
$$
and  when $n_1, n_2 \rightarrow \infty$, we get
\begin{equation}
  T_N=-2\log L_1 \Rightarrow \chi^2_{\frac{1}{2}p(p+1)}\label{classic}
\end{equation}
under $ H_0$. Of cause, in this limit scheme, the data dimension $p$
is held fixed.

However, employing this  $\chi^2$ limit
distribution for  dimensions  like 30 or 40, increases dramatically
the  size of the test. For instance, simulations in
\S\ref{sec:simul2} show that, for dimension and sample sizes $(p,
n_1, n_2)= (40, 800,400)$, the test size  equals 21.2\% instead of
the nominal 5\%  level. The result is worse for the case of $(p,
n_1, n_2)= (80, 1600,800)$, leading to a 49.5\% test size. The
reason for this fail  of  classical LR  test is the following.
Modern RMT indicates that when both dimension and sample size are
large, the likelihood ratio statistic $T_N$ drifts to infinity
almost surely. Therefore, the classical $\chi^2$ approximation leads
to many false rejections of $H_0$ in case of high-dimensional data.

Based on recent CLT for linear spectral statistics of $F$-matrices
from RMT, we propose a correction to this LR  test  in
~\S\ref{sec:test2}. Although this corrected test is constructed
under the asymptotic scheme $n_1 \wedge n_2 \rightarrow +\infty$,
$y_{n_1}={p}/{n_1} \rightarrow y_1 \in (0, 1)$,
$y_{n_2}={p}/{n_2}\rightarrow y_2 \in (0, 1)$, simulations
demonstrate an overall correct behavior including  small or moderate
dimensions  $p$. For example, for the above cited dimension and
sample sizes $(p, n_1, n_2)$, the sizes of the  corrected test equal
5.6\% and 5.2\%, respectively,  both  close to the nominal 5\%
level.

Related works include \citet{LedoitWolf02}, \citet{Sriv05} and
\citet{Schott07}. These authors propose several procedures in the
high-dimensional setting for testing that i) a covariance matrix is
an identity matrix, proportional to an identity matrix (spherecity)
and is a diagonal matrix or ii) several covariance matrices are
equal. These procedures have the following common feature: their
construction involves some well-chosen distance function between the
null and the alternative hypotheses and rely on the first two
spectral  moments, namely the statistics  tr$S_k$ and  tr$S_k^2$
from sample covariance matrices $S_k$. Therefore, the procedures
proposed by these authors are different from the likelihood-based
procedures we consider here. Another  important difference  concerns
the Gaussian assumption on the random variables used in  all these
references. Actually, for testing the equality between  two
covariance matrices,
the correction proposed in
this paper  applies  equally for non-Gaussian and high-dimensional
data leading to a valid  pseudo-likelihood test.

The rest of the paper is organized as following. Preliminary and
useful RMT results are recalled in \S\ref{sec:useful}. In
\S\ref{sec:test1} and \S\ref{sec:test2}, we introduce our results
for the two tests above. Proofs and technical derivations are
postponed to the last section.

\section{Useful results from the random matrix theory }
\label{sec:useful}

We first recall several results from RMT, which will be useful for
our corrections to tests.
For any $p\times p$ square matrix $M$  with real eigenvalues
$\left(\lambda_i^M\right)$, $F_n^M$ denotes the  empirical spectral
distribution (ESD) of $M$, that is,
$$
F_n^M(x) = \frac{1}{p}\sum\limits_{i=1}^{p}\textbf{ 1}_{\lambda_i^M
  \leq x}, \quad \quad x \in \mathbb{R}.
$$
We will consider random matrix $M$ whose ESD $F_n^M$ converges
(in a sense to be precised ) to a limiting spectral distribution
(LSD) $F^M$. To make statistical inference about a parameter $\theta = \int f
(x)dF^M(x)$,   it is natural to use the estimator
$$\widehat{\theta}= \int f (x)dF_n^M(x)=
\frac{1}{p}\sum\limits_{i=1}^p f(\lambda_i^M),
$$
which is a  so-called linear spectral statistic (LSS) of the random
matrix $M$.

\subsection{CLT for LSS of a high-dimensional sample covariance matrix }

Let  $ \{\xi_{ki} \in \mathbb{C}, i, k = 1, 2, \cdots \}$ be a
double array of $i.i.d.$ complex variables with mean 0 and variance
1. Set $\xi_i = (\xi_{1i}, \xi_{2i},\cdots , \xi_{pi})^{T}$ ,  the
vectors $ (\xi_1, \cdots,\xi_{n}) $ is considered as  an $i.i.d$
sample from some $p$-dimensional distribution with mean $0_p$ and
covariance matrix $I_p$. Therefore the sample covariance matrix is
\begin{equation}
  S_n={1\over{n}}\sum\limits_{i=1}^{n}\xi_i\xi_i^*.\label{Sn}
\end{equation}

For $0<\theta\le 1$, let $a(\theta)=(1-\sqrt{\theta})^2$
and $b(\theta)=(1+\sqrt{\theta})^2$. The Mar\v{c}enko-Pastur
distribution of index $\theta$, denoted as $F^\theta$, is the
distribution on $[a(\theta), b(\theta)]$ with the following density
function 
\[  g_\theta(x) =\frac{1}{2\pi \theta x} \sqrt{
  [b(\theta)-x][x-a(\theta)]},\quad  a(\theta)\le x\le b(\theta).
\] 

Let $$y_n=\frac{p}n \rightarrow y \in (0, 1)$$
and $F^y, F^{y_n}$ be the Mar\v{c}enko-Pastur law of index
$y~ \mbox{and}  ~ y_n$, respectively.
Let  $\mathcal{U}$ be an open set of the complex plane,
including   $[ I_{(0, 1)}(y) a(y),  b(y)] $,
and $\mathcal{A}$   be the set of analytic functions $f :
\mathcal{U} \mapsto \mathbb{C}.$  We consider the empirical process
$G_n : = \{G_n(f)\}$ indexed by $\mathcal{A}$ ,
\begin{equation}
  G_n(f)= p\cdot \int_{-\infty}^{+\infty} f(x)\left[F_n-
    F^{y_n}\right] (dx), ~~~~~~\quad f \in  \mathcal{A},\label{Gdef}
\end{equation}
where $F_n$ is the ESD of $S_n$. The following theorem will play a
fundamental role in next derivations, which is a specialization of a
general theorem from \citet{B04}  (Theorem 1.1).

\begin{theorem}
  Assume that
  $ f_1, \cdots ,f_k \in \mathcal{A}$, and  $\{\xi_{ij}\}$  are  $i.i.d.$
  random variables, such that  $E\xi_{11}=0,~ E{|\xi_{11}|^2}=1,
  ~E{|\xi_{11}|}^4 < \infty. $ Moreover, $\frac{p}{n} \rightarrow y
  \in (0, 1) $
  as  $n, p \rightarrow \infty.$\\
  Then: \\
  \textit{(i)} Real Case. \quad  Assume  $\{\xi_{ij}\}$ are real and
  $E(\xi_{11}^4)=3.$  Then the random vector $\left(G_n(f_1),\cdots
  ,G_n(f_k) \right)$  weakly converges to a $k$-dimensional Gaussian
  vector with mean vector,
  \begin{equation}
    m(f_j)=\frac{f_j\left(a(y)\right)+f_j\left(b(y)\right)}{4} -
    \frac{1}{2\pi}\int_{a(y)}^{b(y)}\frac{f_j(x)}{\sqrt{4y-(x-1-y)^2}}dx,
    \quad j=1,\cdots , k,\label{04mean}
  \end{equation}
  and covariance function
  \begin{equation}
    \upsilon\left(f_j,
    f_\ell\right)=-\frac1{2\pi^2}\oint\oint\frac{f_j(z_1)f_\ell(z_2)}{(\underline{m}(z_1)-\underline{m}(z_2))^2}
    d\underline{m}(z_1)d\underline{m}(z_2),\quad j,\ell \in \{1, \cdots,
    k\}\quad \label{04var}
  \end{equation}\\
  where  ~$\underline{m}(z)\equiv m_{\underline{F}^y}(z)$ is the
  Stieltjes Transform of ~ $\underline{F}^y\equiv (1-y)I_{[0,
      \infty)}+yF^y$. The contours
    in (\ref{04var}) are non overlapping and both contain the support of $F^y$.\\

    \textit{(ii)} Complex Case. \quad Assume $\{\xi_{ij}\}$ are complex
    and $E\xi^2_{11}=0$ ,
    $E(|\xi_{11}|^4)=2$.  Then the conclusion of \textit{(i)} also holds, except the mean vector is zero
    and the covariance function is half of the function given in
    (\ref{04var}).\label{T1.1}
\end{theorem}

It is worth noticing that Theorem 1.1 in \citet{B04} covers more general
sample covariance matrices of  form  $S'_n=T_n^{1/2}S_nT_n^{1/2}$ where $(T_n)$ is a
given sequence of positive-definite Hermitian matrices. In the
``white'' case $T_n\equiv I$ as considered here, in a recent
preprint \citet{PL08}, the authors offer a new extension of the CLT 
where the constraints   $E|\xi_{11}|^4=3$ or 2, as stated above,  are
removed.

\subsection{CLT for LSS of high-dimensional F matrix }

Let  $ \{\xi_{ki} \in \mathbb{C}, i, k = 1, 2, \cdots \}$ and
$\{\eta_{kj} \in \mathbb{C}, j, k = 1, 2, \cdots \}$ are two
independent double arrays of $i.i.d.$ complex variables with mean 0
and variance 1. Write $\xi_i = (\xi_{1i}, \xi_{2i},\cdots ,
\xi_{pi})^{T}$ and $\eta_j = (\eta_{1j}, \eta_{2j}, \cdots,
\eta_{pj})^{T}$. Also, for any positive integers $n_1, n_2$, the
vectors $ (\xi_1, \cdots,\xi_{n_1}) $ and $ (\eta_1, \cdots,
\eta_{n_2}) $ can be thought as independent samples of size $ n_1 $
and $ n_2 $, respectively, from some $p$-dimensional distributions.
Let $S_1$ and $S_2$ be the associated sample covariance matrices, ~
$i. e.$
$$
S_{1}={1\over{n_1}}\sum\limits_{i=1}^{n_1}\xi_i\xi_i^* \quad
\mbox{and}\quad
S_{2}={1\over{n_2}}\sum\limits_{j=1}^{n_2}\eta_j\eta_j^*$$ Then, the
following  so-called  F-matrix generalizes  the classical Fisher-statistics for
the present $p$-dimensional case,
\begin{equation}
  V_n=S_{1}S_{2}^{-1}\label{F}
\end{equation}
where $n_2 > p$. Here we use the notation $n=(n_1, n_2)$.

Let
\begin{equation}
  y_{n_1}=\frac{p}{n_1}
  \rightarrow y_1 \in (0, 1),
  ~ y_{n_2}=\frac{p}{n_2} \rightarrow y_2
  \in (0, 1).\label{limitscheme}
\end{equation}
Under suitable moment conditions,  the ESD $F_n^{V_n}$ of $V_n$
has a  LSD  $F_{y_1, y_2}$,
which has a density [See P72 of \cite{B06}], given by
\begin{equation}
  \displaystyle{\ell(x)}=\left\{
  \begin{array}{ll} & \displaystyle{\frac {(1-y_{2})\sqrt{(b-x)(x-a)}}{2\pi x(y_{1}+y_{2}x)},
      \quad ~a \leq x \leq b,}\\[6mm]
    &\displaystyle{0, \quad\quad \quad\quad \mbox{otherwise}.}
  \end{array}
  \right.\label{LSDden}
\end{equation}
where $
a=(1-y_{2})^{-2}\left(1-\sqrt{y_{1}+y_{2}-y_{1}y_{2}}\right)^2
~\mbox{and}~
b=(1-y_{2})^{-2}\left(1+\sqrt{y_{1}+y_{2}-y_{1}y_{2}}\right)^2.$

Similar to previously, let $\widetilde{\mathcal{U}}$ be an open set
of the complex plane,
including  the interval
$$\left[I_{(0,1)}(y_{1})
  \frac{(1-\sqrt{y_{1}})^2}{(1+\sqrt{y_{2}})^2},\quad
  \frac{(1+\sqrt{y_{1}})^2}{(1-\sqrt{y_{2}})^2}\right], $$ and
$\widetilde{\mathcal{A}}$  be the set of analytic functions $f :
\widetilde{\mathcal{U}} \mapsto \mathbb{C}.$  Define the empirical
process $\widetilde{G_n} : = \{\widetilde{G_n}(f)\}$ indexed by
$\widetilde{\mathcal{A}}$
\begin{equation}
  \widetilde{G_n}(f)= p\cdot \int_{-\infty}^{+\infty}
  f(x)\left[F_{n}^{V_n}- F_{y_{n_1}, y_{n_2}}\right] (dx), ~~~~~~f \in
  \widetilde{\mathcal{A}}. \label{Gdef2}
\end{equation}
Here $F_{y_{n_1}, y_{n_2}}$ is the limiting distribution in
(\ref{LSDden}) but with $y_{n_k}$ instead of $y_k, k=1,2.$

Recently, \citet{Z08} establishes  a general  CLT for LSS of
large-dimensional F matrix.  The following theorem is a simplified
one quoted from it, which will play an important role.

\begin{theorem}
  Let  $ f_1, \cdots ,f_k \in \widetilde{\mathcal{A}}$,
  and assume:\\
  For each p,  $(\xi_{ij_1})$  and $(\eta_{ij_2})$ variables are $ i.i.d. $,
  $ 1\leq i\leq p, ~ 1\leq
  j_1 \leq n_1, ~  1\leq j_2 \leq n_2. $ $ E\xi_{11}=E\eta_{11}=0, $
  $E|\xi_{11}|^4=E|\eta_{11}|^4<\infty,$~ $y_{n_1}=\frac{p}{n_1}
  \rightarrow y_1 \in (0, 1),\quad
  y_{n_2}=\frac{p}{n_2} \rightarrow y_2 \in (0, 1).$\\
  Then\\
  \textit{(i)} Real Case. \quad Assume $ (\xi_{ij}) $  and $
  (\eta_{ij}) $ are real, $ E|\xi_{11}|^2=E|\eta_{11}|^2=1 $, then the
  random vector $ \left(\widetilde{G_n}(f_1),
  \cdots,\widetilde{G_n}(f_k)\right) $ weakly converges to a
  k-dimensional Gaussian vector  with the mean vector
  \begin{eqnarray}
    m(f_j)&=&\lim\limits_{r \rightarrow
      1_{+}}\left[(\ref{E1})+(\ref{E1betax})+(\ref{E1betay})\right]\nonumber\\
    &&\quad\frac{1}{4\pi i}\oint_{|\zeta|=1}
    f_j(z(\zeta))\left[\frac{1}{\zeta-{1\over r}}+\frac{1}{\zeta+{1\over
          r}}-\frac{2}{\zeta+{y_2\over
          {hr}}}\right]d\zeta \label{E1}\\
    &&\quad + \frac{\beta\cdot y_1(1-y_2)^2}{2\pi i \cdot
      h^2}\oint_{|\zeta|=1}f_j(z(\zeta))\frac{1}{(\zeta+\frac{y_2}{hr})^3}d\zeta\label{E1betax}\\
    &&\quad + \frac{\beta\cdot y_2(1-y_2)}{2 \pi i \cdot
      h}\oint_{|\zeta|=1}f_j(z(\zeta))\frac{\zeta +
      \frac{1}{hr}}{(\zeta+\frac{y_2}{hr})^3}d\zeta ,\label{E1betay}
    \quad\quad j=1, \cdots, k,
  \end{eqnarray}
  where $
  z(\zeta)=(1-y_2)^{-2}\left[1+h^2+2h\mathcal{R}(\zeta)\right], \quad
  h =\sqrt{y_1+y_2-y_1y_2}, $ $ \beta=E|\xi_{11}|^4-3,  $
  and the covariance function   as $1<r_1<r_2 \downarrow 1$
  \begin{eqnarray}
    \lefteqn{\upsilon(f_j, f_\ell)= \lim\limits_{1 < r_1 < r_2
        \rightarrow
        1_{+}}\left[(\ref{cov1})+(\ref{cov1betax}))\right]}\nonumber\\
    &&  -\displaystyle\frac{1}{2\pi^2}\oint_{|\zeta_2|=1}
    \oint_{|\zeta_1|=1}\frac{f_j(z(r_1\zeta_1))f_\ell(z(r_2\zeta_2))r_1r_2}{(r_2\zeta_2-r_1\zeta_1)^2}
    d\zeta_1d\zeta_2,\label{cov1}\\
    && -\frac{\beta \cdot (y_1+y_2)(1-y_2)^2}{4\pi^2h^2}
    \oint_{|\zeta_1|=1}\frac{f_j\left( z(\zeta_1)
      \right)}{(\zeta_1+\frac{y_2}{hr_1})^2}d\zeta_1
    \oint_{|\zeta_2|=1}\frac{f_\ell\left( z(\zeta_2)
      \right)}{(\zeta_2+\frac{y_2}{hr_2})^2}d\zeta_2
    \label{cov1betax}\\
    && j, \ell \in \{1,\cdots , k\}.\nonumber
  \end{eqnarray}

  \textit{(ii)}  Complex Case. Assume $ (\xi_{ij}) $  and $
  (\eta_{ij}) $ are complex,  $ E(\xi_{11}^{2})=E(\eta_{11}^{2})=0, $
  then the conclusion of
  \textit{(i)}  also holds, except the means are $ \lim\limits_{r
    \rightarrow 1_{+}}\left[(\ref{E1betax})+(\ref{E1betay})\right] $ and
  the covariance function is  $ \lim\limits_{1 < r_1 < r_2 \rightarrow
    1_{+}}\left[\displaystyle\frac{1}{2}
    \cdot(\ref{cov1})+(\ref{cov1betax})\right], $ where $
  \beta=E|\xi_{11}|^4-2. ~
  $
  \label{T2.1}
\end{theorem}

We should point out that  Zheng's CLT
for $F$-matrices covers more general situations then those cited in
Theorem~\ref{T2.1}. In particular,  the fourth-moments 
$E|\xi_{11}|^4$ and $E|\eta_{11}|^4$ can be different.

The following lemma will be used in \S\ref{sec:test2} for an
application of Theorem \ref{T2.1}  to obtain the formula 
(\ref{testE})  and (\ref{testVar}).

\begin{lemma}
  For the function
  $f(x)=\log(a+bx),\quad x\in \mathbb{R}, \quad a, b>0 $, let
  $(c,d)$ be the unique solution to the equations $$\left\{
  \begin{array}{llll}
    c^2+d^2=a(1-y_2)^2+b(1+h^2),\\
    cd=bh,\\
    0 <d< c.\\
  \end{array} \right.$$
  Analogously,  let  $\gamma, \eta$ be the constants similar to $(c,
  d)$ but for the function $g(x)=\log(\alpha+\beta x), \quad \alpha
  > 0,\quad \beta > 0.$
  Then,  the mean and covariance functions in (\ref{E1}) and
  (\ref{cov1}) equal to
  \begin{eqnarray*}
    m(f)&=&\frac{1}{2}\log\frac{(c^2-d^2)h^2}{(ch-y_2d)^2},\\
    \upsilon(f, g)& =&
    2bhd^{-1}c^{-1}\log{\frac{c\gamma}{c\gamma-d\eta}}.
  \end{eqnarray*}
  \label{lem1}
\end{lemma}

\section{Testing the hypothesis that a high-dimensional covariance
  matrix is equal to a given  matrix } \label{sec:test1}

To test the hypothesis   $H_0 : ~\Sigma_p  =I_p$, let be   the
sample covariance matrix \textbf{S}  and likelihood ratio statistic
$T_n$  as  defined in (\ref{S}) and (\ref{singlVclassict}),
respectively.  For $\xi_i=\textbf{x}_i-\mu_p,$ the array
$\{\xi_i\}_{i=1, \cdots, n}$ contains  $p$-dimensional standard
normal variables under $H_0$. Let
$$
\textbf{S}_n={1\over{n}}\sum\limits_{i=1}^{n}\xi_i\xi_i^*.
$$
and  $$\widetilde{L^*}=\mbox{tr} \textbf{S}_n -
\log|\textbf{S}_n|-p.$$

\begin{theorem}
  Assuming that the conditions of Theorem \ref{T1.1} hold, $L^*$ is
  defined as (\ref{L*})
  and $g(x)= x-\log x-1$.  Then,  under $H_0$ and when $n\rightarrow\infty$
  \begin{equation}
    \widetilde{T_n}=\upsilon(g)^{-\frac{1}{2}}\left[ L^*-p \cdot
      F^{y_n}(g)- m(g)\right] \Rightarrow N \left( 0,
    1\right),
    \label{singlsta}
  \end{equation}
  where $F^{y_n}$ is the Mar\v{c}enko-Pastur law of  index  $ y_n$.
\end{theorem}
\begin{proof}
  Because the difference between   $\textbf{S}$
  and $ \textbf{S}_n$  is a rank-1 matrix,
  $\textbf{S}$ and $\textbf{S}_n$ have the same LSD. So, $L^*$
  and $\widetilde{L^*}$ have the same asymptotic
  distribution. We also have
  \begin{eqnarray*}
    \widetilde{L^*}&=& \mbox{tr}\textbf{S}_n - \log|\textbf{S}_n|-p\\
    &=& \sum\limits_{i=1}^{p} \left(\lambda_i^{\textbf{s}_n}-\log
    \lambda_i^{\textbf{s}_n}-1\right) = p \cdot \int (x- \log x-1)
    dF_n(x)\\
    &=&p \cdot \int g(x) d\left(F_n(x)-F^{y_n}(x)\right) +p \cdot
    F^{y_n}(g),
  \end{eqnarray*}
  so that
  \begin{equation}
    G_n(g)=\widetilde{L^*}-p \cdot F^{y_n}(g).\label{singlESD-pLSD}
  \end{equation}
  By Theorem \ref{T1.1}, $G_n(g)$  weakly converges  to a Gaussian
  vector  with the mean
  \begin{equation}
    m(g)=-\frac{\log{(1-y)}}{2}\label{singlmean}
  \end{equation}
  and variance
  \begin{equation}
    \upsilon(g)=-2\displaystyle\log{(1-y)}-2y.\label{singlvar}
  \end{equation}
  for the real case, which are calculated in  \S\ref{sec:proofs}. For
  the complex case, the mean $m(g)$ is zero and the variance is half
  of $ \upsilon(g). $ Then, by (\ref{singlESD-pLSD}) we arrive at
  \begin{equation}
    \widetilde{L^*}-p\cdot F^{y_n}(g)~\Rightarrow~ N\left(m(g),
    \upsilon(g)\right),\label{1connec}
  \end{equation}
  where
  \begin{equation}F^{y_n}(g)=1-\frac{y_n-1}{y_n}\log{(1-y_n)}\label{singllimit}
  \end{equation}
  can be calculated by the density of LSD of sample covariance matrix
  in \S\ref{sec:proofs}. Because $\widetilde{L^*}$ and $ L^*$ have the
  same asymptotic distribution and (\ref{1connec}), finally we get
  \begin{eqnarray}
    \widetilde{T_n}=\upsilon(g)^{-\frac{1}{2}}\left[L^*-p \cdot
      F^{y_n}(g)- m(g)\right] \Rightarrow N \left( 0, 1\right).\nonumber
  \end{eqnarray}
\end{proof}

\subsection{Simulation study I}
\label{sec:simul1} For different values of $(p, n)$, we compute the
realized sizes of traditional likelihood ratio test (LRT) and the
corrected likelihood ratio test (CLRT) proposed previously. The
nominal test level  is set to be $\alpha= 0.05$, and for each $(p,
n)$, we run  10,000 independent replications with real Gaussian
variables. Results are given in Table~1 and Figure~1 below.

\begin{table}[hp]
  \begin{center}
    \begin{tabular}{|l|ccc|cc|}
      \hline & \multicolumn{3}{c|}{CLRT }&  \multicolumn{2}{c|}{LRT } \\[1mm]
             {(p, $n$ ) } & { Size} &Difference with 5\% & {Power} & {Size}& { Power}
             \\[1mm]\hline
             (5, 500) & 0.0803&0.0303&0.6013 &0.0521&0.5233 \\[2mm]
             (10, 500) & 0.0690&0.0190&0.9517 &0.0555&0.9417  \\[2mm]
             (50, 500) & 0.0594&0.0094&1 &0.2252&1  \\[2mm]
             (100, 500) & 0.0537&0.0037&1 &0.9757&1  \\[2mm]
             (300, 500)  & 0.0515&0.0015&1  &1&1  \\[2mm]
             \hline
    \end{tabular}
  \end{center}
  \caption{Sizes and powers of the traditional LRT and the corrected
    LRT, based on 10,000 independent applications with real Gaussian
    variables. Powers are estimated under the alternative
    $\Sigma_p = \mathop{\mathrm{diag}}(1, 0.05, 0.05, 0.05,\ldots)$.}
\end{table}

\begin{figure}[hp]
  \begin{center}
    \includegraphics[width=9cm]{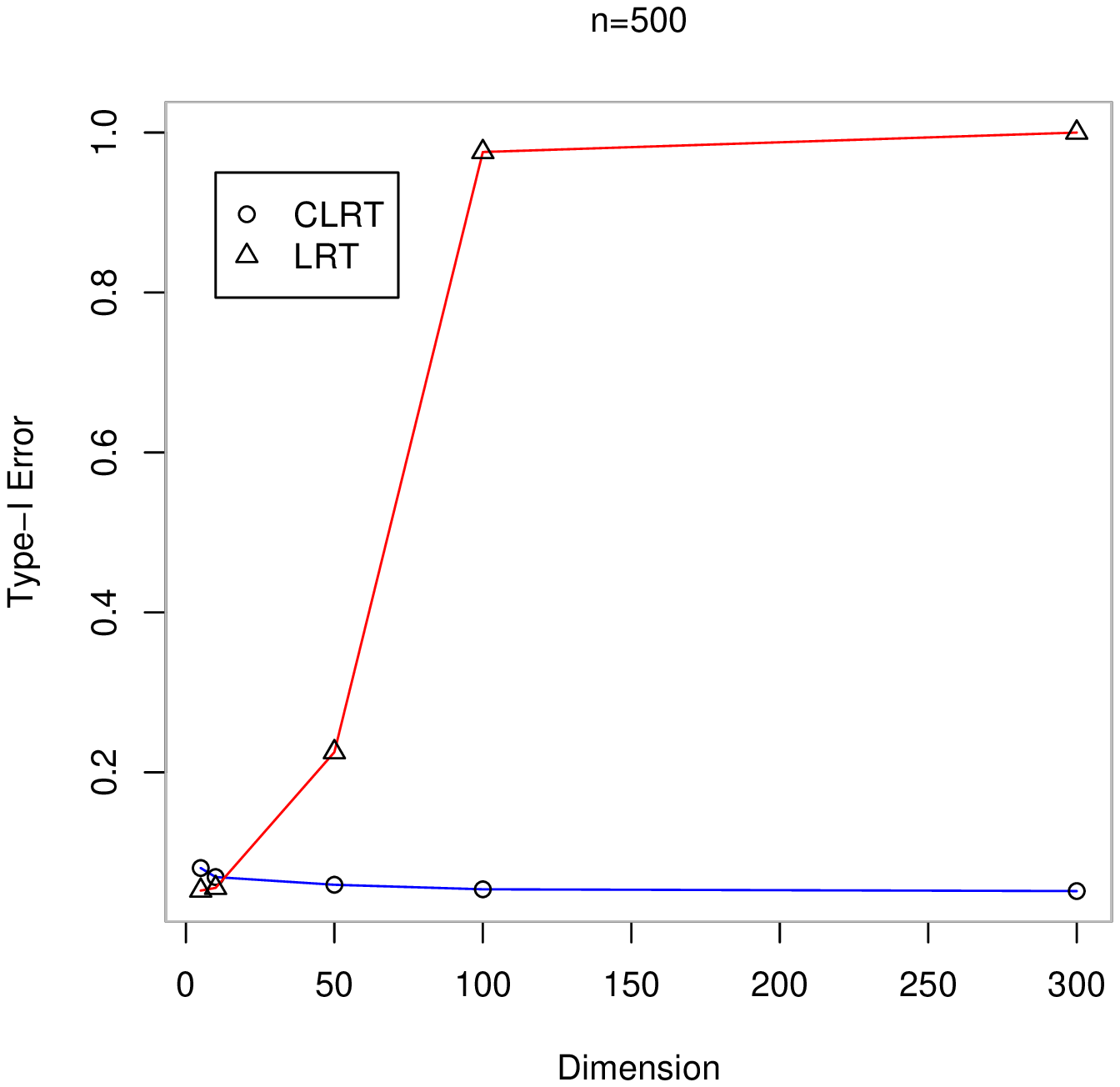}\\
    \caption{Realized sizes of the traditional LRT and the corrected
      LRT for different dimensions $p$ with real Gaussian variables.
      10 000 independent runs with 5\%  nominal level and sample size
      $n=500$.\label{1}}
  \end{center}
\end{figure}

As seen from  Table~1, the traditional LRT always rejects $H_0$ when
$p$ is large, like  $p=100$ or 300, while the sizes produced by the
corrected LRT perfectly matches  the nominal level. For moderate
dimensions like $p=50$,  the corrected LRT still performs correctly
while the traditional  LRT has a size much higher than 5\%.

\section{Testing the equality of two high-dimensional covariance matrices }
\label{sec:test2}

Let $ ({\bf x}_i), ~  i=1, \cdots, n_1 $ and $ ({\bf y}_j) , ~  j=1,
\cdots, n_2$ be observations from two normal populations
$N(\mu_k,\Sigma_k),~ k=1,2$, respectively. We examine the test
defined in (\ref{a0}) and (\ref{L1AB}).  The aim is to find a good
scaling of the LR statistic $T_N$, such that the scaled statistic
weakly converges to some limiting distribution. Let
$$
\xi_i= \Sigma^{-\frac{1}2} (\textbf{x}_i-\mu_1),\quad
\eta_i=\Sigma^{-\frac{1}2} (\textbf{y}_i-\mu_2)
$$
where $\Sigma=\Sigma_1=\Sigma_2$ denotes the common covariance
matrix under $H_0$. Note that in a strict sense, the vectors
$(\textbf{x}_i), (\textbf{y}_i)$ and the matrices $\Sigma, \Sigma_1,
\Sigma_2$ depend on $p$. However we do not  signify this dependence
in  notations for ease of statements. Due to Gaussian assumption,
the arrays $(\xi_i)_{i=1, \cdots, n_1}$  and  $(\eta_j)_{j=1,
  \cdots, n_2}$ contain i.i.d. $N(0, 1)$ variables, for which we can
apply Theorem \ref{T2.1}.

Let
\begin{eqnarray*}
  S_{1}&=&{1\over{n_1}}\sum\limits_{i=1}^{n_1}\xi_i\xi_i^*
  =\Sigma^{-\frac{1}2} C \Sigma^{-\frac{1}2} \\
  S_{2}&=&
  {1\over{n_2}}\sum\limits_{j=1}^{n_2}\eta_j\eta_j^*=\Sigma^{-\frac{1}2}
  D \Sigma^{-\frac{1}2},
\end{eqnarray*}
where
\begin{eqnarray}
  C&=&\frac{1}{n_1}\sum\limits_{i=1}^{n_1}({\bf
    x}_{i}-\mu_1)({\bf x}_{i}-\mu_1)^{*}, \nonumber\\
  D&=&\frac{1}{n_2}\sum\limits_{j=1}^{n_2}({\bf y}_{j}-\mu_2)({\bf
    y}_{j}-\mu_2)^{*}.\nonumber
\end{eqnarray}
Note that
\begin{equation}
  V_n=S_{1}S_{2}^{-1}\nonumber
\end{equation}
forms a random F-matrix and we have
\begin{equation}
  \widetilde{L_1}=\frac{\left|S_{1}\right|^{\frac{n_1}{2}}\cdot
    \left|S_{2}\right|^{\frac{n_2}{2}}}
            {\left|c_1S_{1}+c_2S_{2}\right|^{\frac{N}{2}}}=\frac{\left|C\right|^{\frac{n_1}{2}}\cdot
              \left|D\right|^{\frac{n_2}{2}}}
            {\left|c_1C+c_2D\right|^{\frac{N}{2}}}.\label{S1S2CD}
\end{equation}

\begin{theorem} \label{T4.1}
  Assuming that  the conditions of Theorem \ref{T2.1} hold under $H_0$, $L_1$ as
  defined in (\ref{L1AB}) and
  \[
  f(x)=\log(y_{n_1}+y_{n_2}x)-\frac{y_{n_2}}{y_{n_1}
    +y_{n_2}}\log x-\log(y_{n_1}+y_{n_2}).
  \]
  Then, under $H_0$ and as $n_1\wedge n_2\rightarrow \infty$,
  \begin{equation}
    \widetilde{T_N}=\upsilon(f)^{-\frac{1}{2}}\left[
      -\displaystyle\frac{2\log L_1}{N}-p \cdot F_{y_{n_1},y_{n_2}}(f)-
      m(f)\right] \Rightarrow N \left( 0, 1\right).\label{LST}
  \end{equation}
\end{theorem}
\begin{proof}
  As  $A-C$  and $ B-D$  are rank-1 random matrices, $AB^{-1}$  and
  $CD^{-1}$  have the same LSD. Also by (\ref{S1S2CD}),
  $\widetilde{L_1}$ and $ L_1$  have the same asymptotic distribution.
  Because
  \begin{eqnarray*}
    -\frac{2}{N}\log\widetilde{L_1}
    &=&-\frac{2}{N}
    \log\left(
    \frac{\left|S_{1}\right|^{\frac{n_1}{2}}\cdot
      \left|S_{2}\right|^{\frac{n_2}{2}}}
         {\left|c_1S_{1}+c_2S_{2}\right|^{\frac{N}{2}}}
         \right)
         \\
         &=&\log{|c_1V_n^{-1}+c_2|}- c_1\cdot \log|V_n^{-1}|\\
         &=&\sum\limits_{i=1}^p \log(c_1\lambda^{V_n}_i+c_2)-c_1\cdot
         \log(\lambda^{V_n}_i)\\
         &=& p \cdot \int\left[ \log(c_1x+c_2)-c_1\cdot \log(x)
           \right]dF_n^{V_n}(x).
  \end{eqnarray*}
  Define $f(x)=\log(c_1x+c_2)-c_1\cdot \log(x)$, by
  $c_1=\frac{n_1}{N}=\frac{y_{n_2}}{y_{n_1}+y_{n_2}}$ and
  $c_2=\frac{n_2}{N}=\frac{y_{n_1}}{y_{n_1}+y_{n_2}}$, also it can be
  written as
  \begin{equation} f(x)=\log(y_{n_1}+y_{n_2}x)-\frac{y_{n_2}}{y_{n_1}
      +y_{n_2}}\log x-\log(y_{n_1}+y_{n_2}).\label{f(x)}
  \end{equation}
  From
  \begin{eqnarray}
    -\frac{2\log \widetilde{L_1}}{N}&=&p \cdot \int f(x)dF_n^{V_n}(x)\nonumber\\
    &=& p \cdot \int f(x) d\left(F_n^{V_n}(x)-F_{y_{n_1},
      y_{n_2}}(x)\right) +p \cdot F_{y_{n_1}, y_{n_2}}(f),\nonumber
  \end{eqnarray}
  we get
  \begin{equation}
    \widetilde{G_n}(f)=-\frac{2\log \widetilde{L_1}}{N}
    -p \cdot F_{y_{n_1}, y_{n_2}}(f).\label{ESD-pLSD}
  \end{equation}

  By Theorem \ref{T2.1}, $ \widetilde{G_n}(f)$  weakly converges to a
  Gaussian vector with mean
  \begin{equation}
    m(f)=
    \frac{1}{2}\left[\log\left(\frac{y_1+y_2-y_1y_2}{y_1+y_2}\right)
      -\frac{y_1}{y_1+y_2}\log(1-y_2)-\frac{y_2}{y_1+y_2}\log(1-y_1)\right]\label{testE}
  \end{equation}
  and variance
  \begin{equation}
    \upsilon(f)=-\frac{2y_2^2}{(y_1+y_2)^2}\log(1-y_1)-\frac{2y_1^2}{(y_1+y_2)^2}\log(1-y_2)
    -2\log\frac{y_1+y_2}{y_1+y_2-y_1y_2}\label{testVar}
  \end{equation}
  for the real case, which are  calculated by Lemma \ref{lem1} in
  \S\ref{sec:proofs}. For the complex case, the mean $ m(f) $ is zero
  and the variance is half of $ \upsilon(f) $. In other words,
  \begin{eqnarray}
    -\frac{2\log \widetilde{L_1}}{N}-p\cdot F_{y_{n_1}, y_{n_2}}(f)
    &\Rightarrow & N\left(m(f), \upsilon(f)\right), \label{2connec}
  \end{eqnarray}
  where
  \begin{eqnarray*}
    \displaystyle{F_{y_{n_1}, y_{n_2}}(f)} &=&
    \displaystyle{\frac{-(y_{n_1}+y_{n_2}-y_{n_1}y_{n_2})}{y_{n_1}y_{n_2}}\log{(y_{n_1}+y_{n_2}-y_{n_1}y_{n_2})}}\\[5mm]
    \quad
    &+&\displaystyle{\frac{(y_{n_1}+y_{n_2}-y_{n_1}y_{n_2})}{y_{n_1}y_{n_2}}\log{(y_{n_1}+y_{n_2})}}
    +\displaystyle{\frac{y_{n_1}(1-y_{n_2})}{y_{n_2}(y_{n_1}+y_{n_2})}\log{(1-y_{n_2})}}\nonumber\\
    \quad\quad\quad&+&\displaystyle{\frac{y_{n_2}(1-y_{n_1})}{y_{n_1}(y_{n_1}+y_{n_2})}\log{(1-y_{n_1})}},\nonumber
    \label{limit}
  \end{eqnarray*}
  is derived by use of the density of $F_{y_{n_1}, y_{n_2}}$ in
  \S\ref{sec:proofs}. Because $\widetilde{L_1}$and $ L_1$ have the
  same asymptotic distribution and by (\ref{2connec}),  we get by letting
  $n_1\wedge n_2\rightarrow\infty$,
  $$
  \widetilde{T_N}=\upsilon(f)^{-\frac{1}{2}}\left[
    -\displaystyle\frac{2\log L_1}{N}-p \cdot  F_{y_{n_1}, y_{n_2}}(f)-
    m(f)\right] \Rightarrow N \left( 0, 1\right).
  $$
\end{proof}

\subsection{Simulation study II}
\label{sec:simul2}

For different values of $(p, n_1,n_2)$, we compute the realized
sizes of the traditional LRT and the corrected LRT   with 10,000
independent replications. The nominal  test level is $\alpha=  0.05$
and we use real Gaussian variables. Results are summarized in Table
2  and Figure 2.

\begin{table}
  \begin{center}
    \begin{tabular}{|l|ccc|cc|}
      \hline \multicolumn{6}{|c|}{(y1, y2)=(0.05, 0.05)~~}\\
      \hline & \multicolumn{3}{c|}{CLRT  }&  \multicolumn{2}{c|}{LRT } \\[1mm]
             {(p, $n_1$,  $n_2$ ) } & { Size} &Difference with 5\% &
             {Power} & {Size}& { Power}
             \\[1mm]\hline
             (5, 100, 100)  & 0.0770&0.0270&1  &0.0582&1  \\[2mm]
             (10, 200, 200) & 0.0680&0.0180&1 &0.0684&1  \\[2mm]
             (20, 400, 400) & 0.0593&0.0093&1 &0.0872&1  \\[2mm]
             (40, 800, 800) & 0.0526&0.0026&1 &0.1339&1  \\[2mm]
             (80, 1600, 1600)   &0.0501 &0.0001&1&0.2687&1  \\[2mm]
             (160, 3200, 3200)   &0.0491 &-0.0009&1&0.6488&1  \\[2mm]
             (320, 6400, 6400)   &0.0447 &-0.0053&0.9671&1&1  \\[2mm]
             \hline
    \end{tabular}\\[2mm]
    \begin{tabular}{|l|ccc|cc|}
      \hline \multicolumn{6}{|c|}{(y1, y2)=(0.05, 0.1)~~}\\
      \hline & \multicolumn{3}{c|}{CLRT  }&  \multicolumn{2}{c|}{LRT } \\[1mm]
             {(p, $n_1$,  $n_2$ ) } & { Size} &Difference with 5\% &
             {Power} & {Size}& { Power}
             \\[1mm]\hline
             (5, 100, 50)   & 0.0781&0.0281&0.9925 &0.0640&0.9849  \\[2mm]
             (10, 200, 100) & 0.0617&0.0117&0.9847 &0.0752&0.9904  \\[2mm]
             (20, 400, 200) & 0.0573&0.0073&0.9775 &0.1104&0.9938  \\[2mm]
             (40, 800, 400) & 0.0561&0.0061&0.9765 &0.2115&0.9975  \\[2mm]
             (80, 1600, 800)&0.0521 &0.0021&0.9702 &0.4954&0.9998 \\[2mm]
             (160, 3200, 1600)   &0.0520&0.0020&0.9702&0.9433&1  \\[2mm]
             (320, 6400, 3200)   &0.0510 &0.0010&1&0.9939&1  \\[2mm]
             \hline
    \end{tabular}
  \end{center}
  \caption{Sizes and powers of the traditional LRT  and the corrected
    LRT  based on 10,000 independent replications using real
    Gaussian variables.  Powers are estimated under the
    alternative  $\Sigma_1\Sigma_2^{-1}= \mbox{ diag} (3,1,1,1,\cdots)$.
    Upper:  $y_1=y_2=0.05$.~~
    Bottom:   $y_1=0.05,~ y_2=0.1$.\label{tab:23}}
\end{table}

\begin{figure}[hb]
  \begin{center}
    \includegraphics[width=7cm]{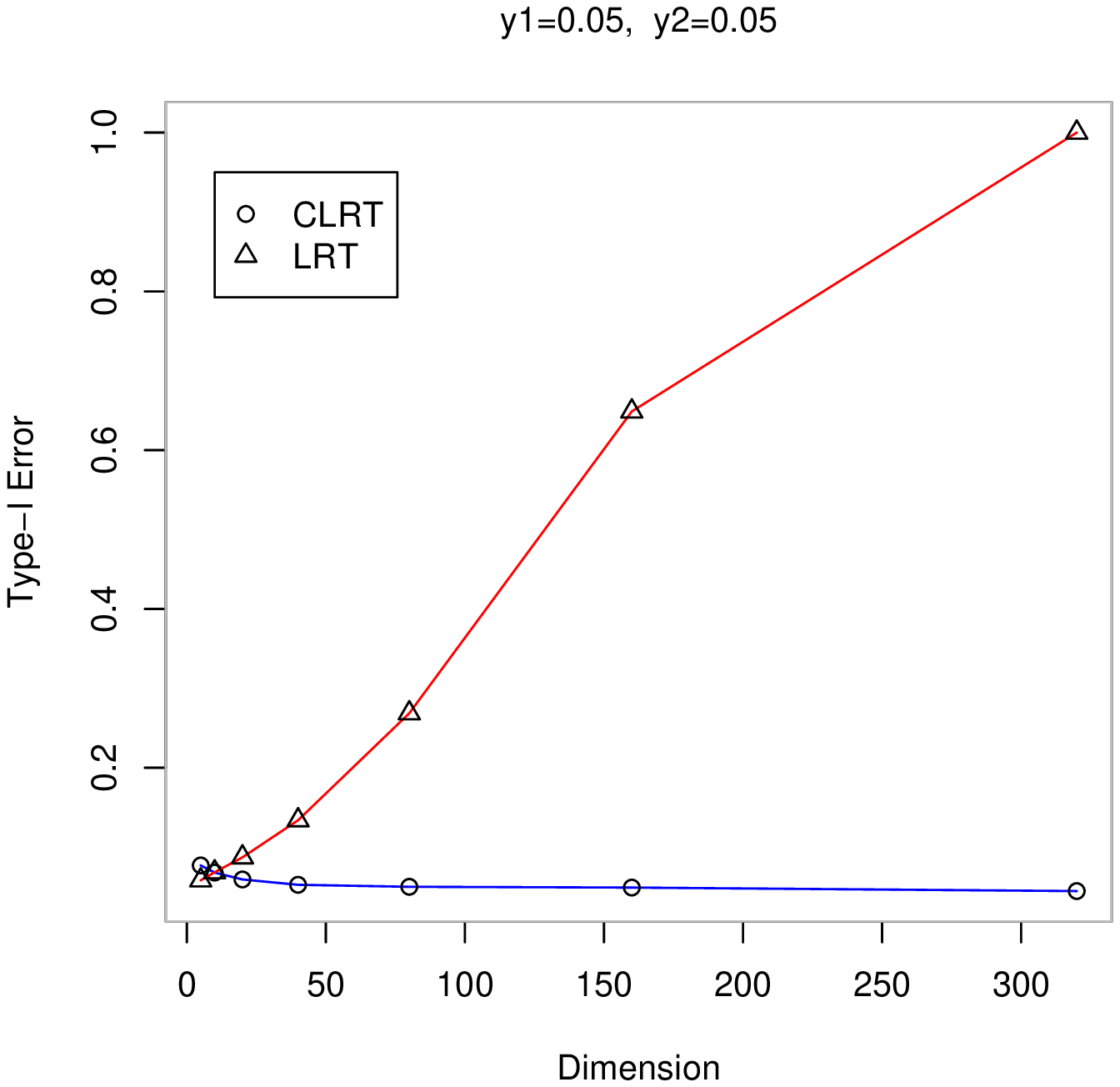}\quad
    \includegraphics[width=7cm]{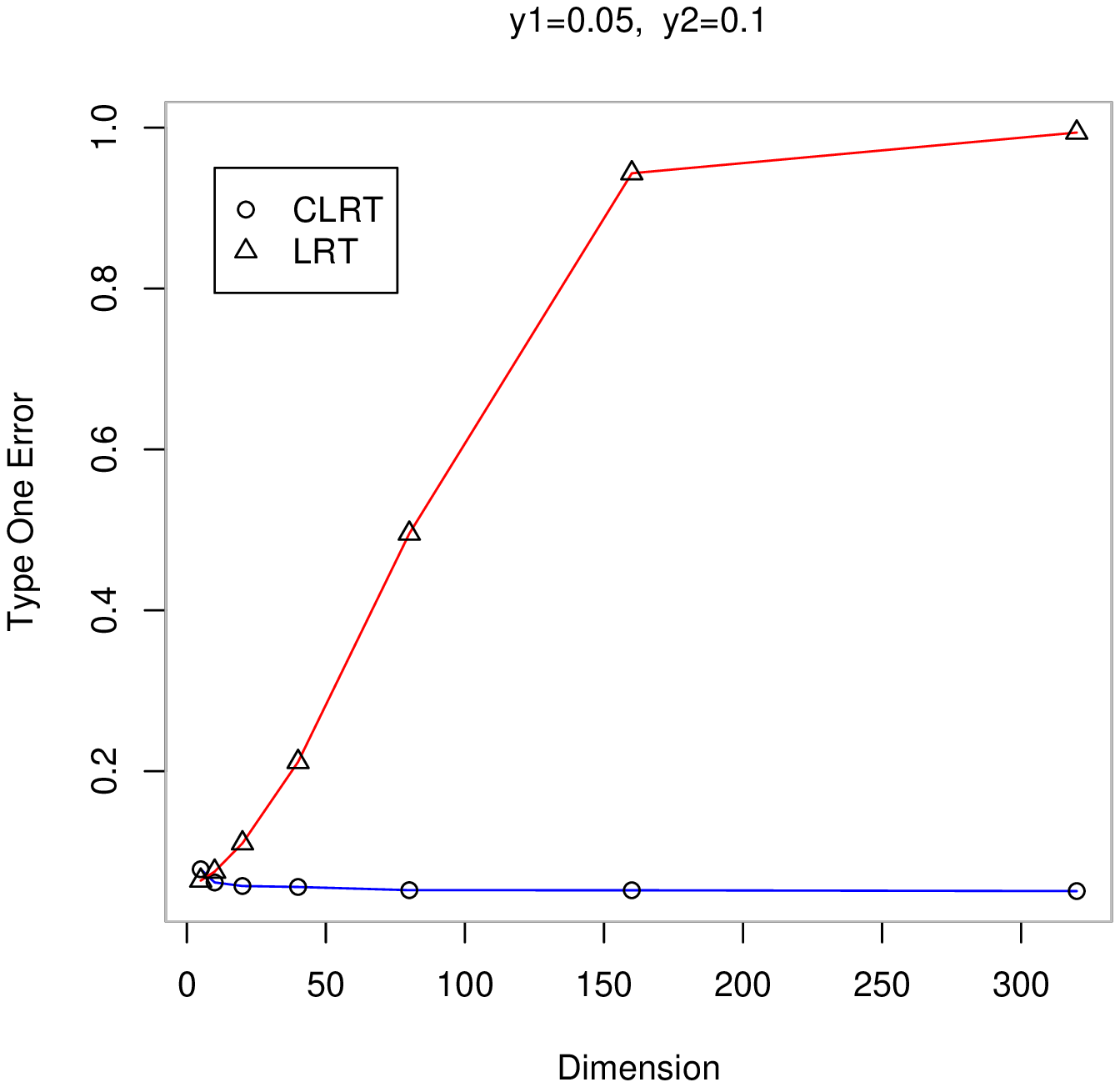}
    \caption{Sizes of  the traditional LRT  and the corrected
      LRT  based on 10,000 independent replications using real
      Gaussian variables.
      Left:  $y_1=y_2=0.05$.~~
      Right:   $y_1=0.05,~ y_2=0.1$.\label{fig:2}}
  \end{center}
\end{figure}

As we can see, when the dimension $p$ increases,   the traditional
LRT leads to a dramatically high test size while the corrected LRT
remains accurate. Furthermore,  for  moderate dimensions like $p=20$
or 40, the sizes of  the traditional LRT are much higher than 5\%,
whereas the ones of corrected LRT  are  very close. By a closer look
at the  column showing the difference with  5\%, we note that  this
difference rapidly decrease as $p$ increases for the corrected test.
Figure~2 gives  a vivid sight of these  comparisons between the
traditional LRT and the corrected LRT  in term of test  sizes.

\subsection{A pseudo-likelihood test for high-dimensional non-Gaussian data}

As said in Introduction, previous related works as
\citet{LedoitWolf02}, \citet{Sriv05}  or \citet{Schott07} all assume
Gaussian variables.  In contrast, Theorem~\ref{T4.1} applies for
general distributions having a fourth moment. For these non Gaussian
data, we consider the corrected LRT as generalized
pseudo-likelihood ratio test (or Gaussian  LRT).

Moreover, the methods proposed by these authors all rely  on an appropriate
normalization of 
the trace of squared difference between two sample covariances
following the  
idea of \citet{B96}.
We believe that their method
would strongly depend on the normality assumption (which was supported  
by simulation results below).
On the other hand, based on
general understanding, the LRT contains much higher information from  
data and its poor performance observed up to now  is just caused by its large bias
when dimension is large. Thus,
from the intuitive understanding, we  
are confined ourselves to modify the LRT.

Let us develop in more details an example. Assume that \textbf{x}
follows  a normalized $t$-distribution with 5 degree of freedom,
that is $ \textbf{x}=\sqrt{\frac{3}{5}} t(5), $ \textbf{x} and
\textbf{y} are i.i.d., hence  $E\textbf{x}=E\textbf{y}=0, $ $
E|\textbf{x}|^2=E|\textbf{y}|^2=1 $ and $
E|\textbf{x}|^4=E|\textbf{y}|^4=9. $  We still  employ the result in
Theorem \ref{T4.1} for the test of equality between two covariance
matrices, where
\begin{eqnarray}
  m_1(f)=&
  \displaystyle{\frac{1}{2}\Big[\log\left(\frac{y_1+y_2-y_1y_2}{y_1+y_2}\right)
      -\frac{y_1}{y_1+y_2}\log(1-y_2)-\frac{y_2}{y_1+y_2}\log(1-y_1)}\nonumber\\[2mm]
    &\quad\quad\quad\quad\quad\quad\quad\quad
    +\displaystyle{\frac{6y_1^2y_2}{(y_1+y_2)^2}+\frac{6y_1y_2^2}{(y_1+y_2)^2}\Big]}\label{m1f}
\end{eqnarray}
and
\begin{equation}
  \upsilon_1(f)=-\frac{2y_2^2}{(y_1+y_2)^2}\log(1-y_1)-\frac{2y_1^2}{(y_1+y_2)^2}\log(1-y_2)
  -2\log\frac{y_1+y_2}{y_1+y_2-y_1y_2}\label{v1f}
\end{equation}
instead of $ m(f) $  and $ \upsilon(f) $ for real case,
respectively. (\ref{m1f}) and (\ref{v1f}) are calculated in
\S\ref{sec:proofs}.

The following Table~\ref{tab:t} summarizes a simulation study where
we compare  this corrected pseudo-LRT with the test proposed in
\citet{Schott07}.  We use  1,000 independent replications with the
above $t$-distributed variables. Again, the nominal  test level is
$\alpha=  0.05$. As we can see, the corrected pseudo-LRT performs
correctly while Schott's test is no more valid here since the
variables are not Gaussian.

\begin{table}[htb]
  \begin{center}
    \begin{tabular}{|l|c|c|}
      \hline \multicolumn{3}{|c|}{(y1, y2)=(0.05, 0.1)~~}\\
      \hline {(p, $n_1$,  $n_2$ ) } & {CLRT Size}  & {Schott's Size}
      \\[1mm]\hline
      (10,100, 200) & 0.067&0.517\\[2mm]
      (20, 200, 400) & 0.065&0.603  \\[2mm]
      (40, 400, 800) & 0.054&0.703 \\[2mm]
      (80, 800, 1600)&0.048 &0.764 \\[2mm]
      (160, 1600, 3200)   &0.045&0.826\\[2mm]
      (320, 3200, 6400)   &0.051&0.854\\[2mm]
      \hline
    \end{tabular}
    \caption{Sizes  of the corrected pseudo-likelihood ration test  and Schott's test for the case
      of $y_1=0.1,~ y_2=0.05$, based on 1,000 independent replications with
      normalized $t$-distributed variables with  5 degrees of freedom. \label{tab:t}}
  \end{center}
\end{table}

\section{Proofs}
\label{sec:proofs}

\subsection*{Proof of (\protect\ref{singlmean})}

By Theorem \ref{T1.1}, for $g(x)= x-\log x-1$, by using the variable
change 
$x=1+y-2\sqrt{y}\cos\theta,~ 0\leq \theta \leq \pi$, we have 
\begin{eqnarray*}
  m(g)&=&\frac{g\left(a(y)\right)+g\left(b(y)\right)}{4} -
  \frac{1}{2\pi}\int_{a(y)}^{b(y)}\frac{g(x)}{\sqrt{4y-(x-1-y)^2}}dx\\
  &=&\frac{y-\log(1-y)}{2}-\frac{1}{2\pi}\int_{0}^{\pi}
  \left[1+y-2\sqrt{y}\cos\theta-\log(1+y-2\sqrt{y}\cos\theta)-1\right]d\theta\\
  &=&\frac{y-\log(1-y)}{2}-\frac{1}{4\pi}\int_{0}^{2\pi}
  \left[y-2\sqrt{y}\cos\theta-\log|1-\sqrt{y}e^{i
      \theta}|^2\right]d\theta\\
  &=&-\frac{\log(1-y)}{2},
\end{eqnarray*}
where $\displaystyle\int_{0}^{2\pi}\log|1-\sqrt{y}e^{i
  \theta}|^2d\theta=0$ is calculated in \cite{B04}.

\subsection*{Proof of (\ref{singlvar})}

For $g(x)=x-\log x -1$, by Theorem \ref{T1.1}, we have
$$\upsilon(g)=-\frac1{2\pi^2}\oint\oint\frac{g(z_1)g(z_2)}{(\underline{m}(z_1)-\underline{m}(z_2))^2}
d\underline{m}(z_1)d\underline{m}(z_2)$$ and
\begin{eqnarray*}
  g(z_1)g(z_2)&=&z_1z_2-z_1\log z_2-z_2\log z_1+\log z_1\log z_2\\
  \quad\quad &&-z_1+\log z_1 -z_2+\log z_2+1.
\end{eqnarray*}
It is easy to see that  $\upsilon(\textbf{1},\textbf{1})=0$, where
\textbf{1} means constant function equals to 1.  For Stieltjes
transform of $F^y$, the following equation is given in \cite{B04},
for $z \in \mathbb{C}^{+}$,
\begin{equation}
  z=-\frac{1}{\underline{m}(z)}+\frac{y}{1+\underline{m}(z)}.
\end{equation}\\
Let  $m_i=\underline{m}(z_i),~~i=1,2$.  For fixed $m_2$,~we
have on a contour enclosed 1, $(y-1)^{-1}$ and -1, but not 0, 
\begin{eqnarray*}
  \displaystyle {\oint \frac{\log \left(z(m_1)\right)}
    {(m_1-m_2)^2}dm_1}
  &=& \displaystyle {\oint \frac{\frac{1}{m_1^2}
      -\frac{y}{(1+m_1)^2}}
    {-\frac{1}{m_1}+\frac{y}{1+m_1}}\frac{1}{( m_1-m_2)}dm_1} \\
  &=&\displaystyle{\oint
    \frac{(1+m_{1})^2-ym_1^2}{ym_1(m_1-m_2)}\left(\frac{-1}{m_{1}+1}
    +\frac{1}{m_1-\frac{1}{y-1}}\right)dm_1}\\
  &=&\displaystyle{2\pi i\cdot\left(
    \frac{1}{m_2+1}-\frac{1}{m_2-\frac{1}{y-1}}\right)}.
\end{eqnarray*}
and
\begin{eqnarray*}
  &\quad&\displaystyle {\oint \frac{-\frac{1}{m_1}
      +\frac{y}{1+m_1}}
    {(m_1-m_2)^2}dm_1}\\
  &=&\displaystyle{y\oint (\frac{1}{1+m_{1}}+\frac{1-y}{y}) \cdot
    [1-(1+m_{1})]^{-1}
    \cdot (m_{2}+1)^{-2} \cdot (1-\frac{m_{1}+1}{m_{2}+1})^{-2}
    dm_{1}}\\
  &=&\displaystyle {y\oint (\frac{1}{1+m_{1}}+\frac{1-y}{y}) \cdot
    \sum\limits^\infty_{j=0}(1+m_1)^j (m_2+1)^{-2}
    \sum\limits^\infty_{\ell=1}\ell(\frac{m_{1}+1}{m_{2}+1})^{\ell-1}dm_1}\\
  &=&\displaystyle{2\pi i\cdot  \frac{y}{(m_2+1)^{2}}}.
\end{eqnarray*}
Then we also get $\upsilon(-z_1+\log z_1, ~\textbf{1})=0$.
Similarly, $\upsilon(\textbf{1}, ~ -z_2+\log z_2)=0$.
Furthermore, 
\begin{eqnarray*}
  \upsilon(z_1,z_2)&=&\displaystyle{\frac {y^2}{\pi
      i}\oint\frac{1}{(m_2+1)^{2}}(\displaystyle\frac{1}{1+m_{2}}+\displaystyle\frac{1-y}{y})
    \displaystyle\sum\limits^\infty_{j=0}(1+m_2)^jdm_2}
  =2y,
\end{eqnarray*}
and 
\begin{eqnarray*}
  \upsilon(z_1, \log z_2)
  &=&\displaystyle{ \frac{y}{\pi i}\oint( \frac{1}{m_2+1}- \frac
    {1}{m_2-1/(y-1)})
    (\frac{1}{1+m_{2}}+\frac{1-y}{y}) \cdot [1-(1+m_{2})]^{-1} dm_2}\\
  &=&\displaystyle{ \frac{y}{\pi i} \oint( \frac{1}{m_2+1}- \frac
    {1}{m_2-1/(y-1)})(\frac{1}{1+m_{2}}+\displaystyle\frac{1-y}{y})
    \sum\limits^\infty_{j=0}(1+m_2)^j dm_2} \\
  &=&2y.
\end{eqnarray*}
By a computation in  \cite{B04}, we know that $ \upsilon(\log z_1,
\log z_2)= -2\log(1-y)$.  Finally, we obtain
\begin{eqnarray*}
  \upsilon(g)&=&\upsilon(z_1, z_2)+\upsilon(\log z_1, \log
  z_2)-2\upsilon(z_1, \log z_2)\\
  &&\quad+\upsilon(-z_1+\log z_1,
  \textbf{1})+\upsilon(\textbf{1},-z_2+\log z_2)+\upsilon(\textbf{1}, \textbf{1})\\
  &=&-2\log (1-y)-2y.
\end{eqnarray*}

\subsection*{Proof of (\ref{singllimit})}

Since  $F^{y_n}$  is the Mar\v{c}enko-Pastur law of index $y_n$,
by using the variable change $x=1+y_n-2\sqrt{y_n}\cos \theta, ~ 0\leq
\theta \leq \pi$ we have
\begin{eqnarray*}
  F^{y_n}(g)&=&\int^{b(y_n)}_{a(y_n)}\frac {x-\log x-1}{2\pi
    xy_n}\sqrt{(b(y_n)-x)(x-a(y_n))} dx\\
  &=&\frac{1}{2\pi
    y_n}\int_0^{\pi}\left[1-\frac{\log(1+y_n-2\sqrt{y_n}\cos\theta)+1}
    {1+y_n-2\sqrt{y_n}\cos\theta}\right]4y_n\sin^2\theta d\theta\\
  &=&\frac{1}{2\pi}\int_0^{2\pi}\left[2\sin^2\theta-\frac{2\sin^2\theta}{1+y_n-2\sqrt{y_n}\cos\theta}
    \left(\log |1-\sqrt{y_n}e^{i\theta}|^2-1\right)\right]d \theta\\
  &=&1-\frac{y_n-1}{y_n}\log (1-y_n),
\end{eqnarray*}
where
$$\displaystyle{\frac{1}{2\pi}\int_0^{2\pi}\frac{2\sin^2\theta}{1+y_n-2\sqrt{y_n}\cos\theta}
  \log |1-\sqrt{y_n}e^{i\theta}|^2d \theta
}=\displaystyle{\frac{y_n-1}{y_n}\log (1-y_n)-1}$$ is calculated in
\cite{B04}.

\subsection*{ Proof of Lemma  \ref{lem1}}

We use the variable change  $x=(1-y_2)^{-2}(1+h^2-2h\cos \theta)$, where
$h=\sqrt{y_1+y_2-y_1y_2}$.  When $c,d $ satisfy
$c^2+d^2=a(1-y_2)^2+b(1+h^2),~ cd=bh,~ 0 < d < c,$ we have
\begin{equation*}
  f(z(\xi))=\log(a+bz(\xi))=\log\left(\frac{\left|c+d\xi\right|^2}{(1-y_2)^2}\right).
\end{equation*}
Similarly,
$$g(z(\xi))=\log(\alpha+\beta
z(\xi))=\log\left(\frac{\left|\gamma+\eta\xi\right|^2}{(1-y_2)^2}\right).$$
Let 
$$
\widetilde{f}(z(\xi))=\log\left(\frac{\left(c+d\xi\right)^2}{(1-y_2)^2}\right)
\quad \mbox{and}
\quad\widetilde{g}(z(\xi))=\log\left(\frac{\left(\gamma+\eta\xi\right)^2}{(1-y_2)^2}\right).
$$
Note that $f(z(\xi))=\Re(\widetilde{f}(z(\xi)))\quad \mbox{and}
\quad g(z(\xi))=\Re(\widetilde{g}(z(\xi)))$. By  Theorem \ref{T2.1},
we have
\begin{eqnarray}
  m(f)&=& \displaystyle{\frac{1}{4\pi
      i}\oint_{|\xi|=1}
    f(z(\xi))\left[\frac{1}{\xi-{1\over r}}+\frac{1}{\xi+{1\over r}}-\frac{2}{\xi+{y_2\over {hr}}}\right]d\xi}\nonumber\\
  &=&\displaystyle{\frac{1}{4\pi}\int^{2\pi}_{0}
    f(z(e^{i\theta}))\left[\frac{1}{e^{i\theta}-{1\over r}}
      +\frac{1}{e^{i\theta}+{1\over r}}-\frac{2}{e^{i\theta}+{y_2\over {hr}}}\right]e^{i\theta}d\theta}\nonumber\\
  &=&\displaystyle{\frac{1}{4\pi}\int^{2\pi}_{0}
    f(z(e^{i\theta}))\left[\frac{1}{e^{-i\theta}-{1\over r}}
      +\frac{1}{e^{-i\theta}+{1\over r}}-\frac{2}{e^{-i\theta}+{y_2\over
          {hr}}}\right]e^{-i\theta}d\theta}\nonumber\\
  &=&\displaystyle\frac{1}{8\pi}\int^{2\pi}_{0}
  f(z(e^{i\theta}))\Bigg\{\left[\frac{1}{e^{i\theta}-{1\over r}}
    +\frac{1}{e^{i\theta}+{1\over r}}-\frac{2}{e^{i\theta}+{y_2\over {hr}}}\right]e^{i\theta}+\nonumber\\
  &&\left[\frac{1}{e^{-i\theta}-{1\over
        r}}+\frac{1}{e^{-i\theta}+{1\over r}}
    -\frac{2}{e^{-i\theta}+{y_2\over {hr}}}\right]e^{-i\theta}\Bigg\}d\theta\nonumber\\
  &=&\frac{1}{8\pi}\Re\Bigg\{\int^{2\pi}_{0}
  \widetilde{f}(z(e^{i\theta}))\Bigg[\left(\frac{1}{e^{i\theta}-{1\over
        r}}
    +\frac{1}{e^{i\theta}+{1\over r}}-\frac{2}{e^{i\theta}+{y_2\over {hr}}}\right)e^{i\theta}+\nonumber\\
    &&\left(\frac{r}{r-e^{i\theta}}+\frac{r}{r+e^{i\theta}}
    -\frac{2hr}{y_2e^{i\theta}+hr}\right)\Bigg]d\theta\Bigg\}\nonumber\\
  &=&\Re\Bigg\{\frac{1}{8\pi i}\oint_{|\xi|=1}
  \widetilde{f}(z(\xi))\Bigg[\left(\frac{1}{\xi-{1\over
        r}}+\frac{1}{\xi+{1\over r}}-\frac{2}{\xi+{y_2\over {hr}}}
    \right)\nonumber\\
    &&+\left(\frac{r}{r-\xi}+\frac{r}{r+\xi}
    -\frac{2hr}{y_2\xi+hr}\right)\xi^{-1}\Bigg]d\xi\Bigg\}\nonumber\\
  &=&\frac{1}{4}\left(\widetilde{f}(z(\frac{1}{r}))+\widetilde{f}(z(-\frac{1}{r}))
  -2\widetilde{f}(z(-\frac{y_2}{hr}))\right)\nonumber\\
  &\rightarrow& ^{r\downarrow1}\frac{1}{4}\left[\widetilde{f}(z(1))+\widetilde{f}(z(-1))-2\widetilde{f}(z(-\frac{y_2}{h}))\right]\nonumber\\
  &=&\frac{1}{2}\log\frac{(c^2-d^2)h^2}{(ch-y_2d)^2}.\nonumber
\end{eqnarray}

Let $ m_j=-\frac{1+hr_j\xi_j}{1-y_2} $,  where $|\xi_j|=1, j=1,2,$ $
r_2\downarrow r_1,$ and $r_1\downarrow 1$. By  Theorem~\ref{T2.1},
we have
$$\upsilon(f, g)=-\frac{1}{2\pi ^2}\oint_{|\xi_2|=1}
\left\{\oint_{|\xi_1|=1}\frac{f(z(r_1\xi_1))}{(r_2\xi_2-r_1\xi_1)^2}\cdot
r_1r_2 d\xi_1\right\}g(z(r_2\xi_2))d\xi_2.
$$
When $ r_1\downarrow 1, \quad -\frac{d}{cr_1} \quad \mbox{and}\quad
0$  are poles.  We can then  choose  $r_1$ so that 
 $-\frac{c}{dr_1}$  is a not a  pole. Then we get
\begin{eqnarray*}
  &&\displaystyle{\oint_{|\xi_1|=1}\frac{\log(a+bz(r_1\xi_1))}{(r_2\xi_2-r_1\xi_1)^2}\cdot
    r_1r_2 d\xi_1}\\[3mm]
  &=&\displaystyle{\oint_{|\xi_1|=1}\frac{\left(\log(a+bz(r_1\xi_1))\right)'}{r_1\xi_1-r_2\xi_2}\cdot
    r_2 d\xi_1}\\[3mm]
  &=&\displaystyle{\oint_{|\xi_1|=1}\Bigg[\frac{bhr_1\xi_1}{(r_1\xi_1-r_2\xi_2)(c+dr_1\xi_1)c}\cdot
      \frac{1}{\xi_1+\frac{d}{cr_1}}}\\
    &&\quad-\frac{bhr_1^{-1}}{(r_1\xi_1-r_2\xi_2)(c+dr_1\xi_1)c}\cdot
    \frac{1}{(\xi_1+\frac{d}{cr_1})\xi_1}\cdot r_2\Bigg]d\xi_1\\[5mm]
  &=&\displaystyle{2\pi i}\left(
  \frac{bhd^{-1}c^{-1}}{\xi_2}-\frac{bhd^{-1}r_2}{d+cr_2\xi_2}\right).
\end{eqnarray*}
So,
$$
\upsilon(f, g) =-\frac{i}{\pi}
\oint_{|\xi_2|=1}\left(\frac{bhd^{-1}c^{-1}}{\xi_2}
-\frac{bhd^{-1}r_2}{d+cr_2\xi_2}\right) \log\left(\alpha+\beta
z(r_2\xi_2)\right)d\xi_2.
$$
Since the function $g(x)=\log(\alpha+\beta x) $ is analytic, when $
r_2>1$  but sufficiently close to 1, we have
$$
\left|g(z(r\xi_2))-g(z(\xi_2))\right| \leq K(r-1),
$$
for some constant $K$. Thus we have
$$
\begin{array}{lll}
  &\displaystyle{\left|\oint_{|\xi|_2=1}
    \left[g(z(r_2\xi_2))-g(z(\xi_2))\right]
    \left(\frac{bhd^{-1}c^{-1}}{\xi_2}-\frac{bhd^{-1}r_2}{d+cr_2\xi_2}\right)d\xi_2\right|}\\[4mm]
  \rightarrow & 0 \quad \mbox{as} \quad r_2 \downarrow 1,
\end{array}
$$
where the estimations are done according to$|\arg(\xi_2)|$ or
$|\arg(\xi_2)-\pi| \leq \sqrt{r_2-1}$ or not. Thus,
$$
\upsilon(f, g)=-\frac{i}{\pi} \oint_{|\xi_2|=1}
g(z(\xi_2))\left(\frac{bhd^{-1}c^{-1}}{\xi_2}-\frac{bhd^{-1}r_2}{d+cr_2\xi_2}\right)d\xi_2+R(r_2)
$$
where $ R(r_2) \rightarrow 0 , \quad\mbox{as}\quad  r_2\downarrow1.$ 
Because $\displaystyle{g(z(\xi_2))=\log\left(\frac{\left|\gamma+\eta\xi_2\right|^2}{(1-y_2)^2}\right),}$
for $\gamma, \eta$ satisfying
$\gamma^2+\eta^2=\alpha(1-y_2)^2+\beta(1+h^2),~ \gamma\eta=\beta h,~
0 < \eta < \gamma,$ and if 
$\displaystyle{\widetilde{g}(z(\xi_2))=\log\left(\frac{\left(\gamma+\eta\xi_2\right)^2}{(1-y_2)^2}\right),}$
we have  $g(z(\xi_2))=\Re\left(\widetilde{g}(z(\xi_2))\right) $.
Therefore, 
\begin{eqnarray*}
  \lefteqn{    \upsilon(f, g)=-\frac{i}{\pi}
    \oint_{|\xi|_2=1}  g(z(\xi_2))
    \left( \frac{bhd^{-1}c^{-1}}{\xi_2} -
    \frac{bhd^{-1}r_2}{d+cr_2\xi_2}      \right)d\xi_2}\\
  &=&\frac{1} {\pi} \int^{2\pi}_{0} g(z(e^{i\theta}))\left(\frac{bhd^{-1}c^{-1}}{e^{i\theta}}
  -\frac{bhd^{-1}r_2}{d+cr_2e^{i\theta}}\right)e^{i\theta}d\theta\\
  &=&^{\theta \rightarrow 2\pi-\theta} ~~\frac{1}{\pi}
  \int^{2\pi}_{0}g(z(e^{i\theta}))\left(\frac{bhd^{-1}c^{-1}}{e^{-i\theta}}
  -\frac{bhd^{-1}r_2}{d+cr_2e^{-i\theta}}\right)e^{-i\theta}d\theta\\
  &=&\frac{1}{2\pi}\int_{0}^{2\pi}g(z(e^{i\theta}))\Bigg[\left(\frac{bhd^{-1}c^{-1}}{e^{i\theta}}
    -\frac{bhd^{-1}r_2}{d+cr_2e^{i\theta}}\right)
    e^{i\theta}
    +bhd^{-1}c^{-1}-\frac{bhd^{-1}r_2}{de^{i\theta}+cr_2}\Bigg]d\theta\\
  &=&\frac{1}{2\pi}\Re\Bigg\{\int_{0}^{2\pi}\widetilde{g}(z(e^{i\theta}))
  \Bigg[\left(\frac{bhd^{-1}c^{-1}}{e^{i\theta}}-\frac{bhd^{-1}r_2}{d+cr_2e^{i\theta}}\right)
    e^{i\theta}
    +bhd^{-1}c^{-1}-\frac{bhd^{-1}r_2}{de^{i\theta}+cr_2}\Bigg]d\theta\Bigg\}\\
  &=&\Re\Bigg\{\frac{1}{2\pi
    i}\oint_{|\xi|_2=1}\widetilde{g}(z(\xi_2))\Bigg[\left(\frac{bhd^{-1}c^{-1}}{\xi_2}-\frac{bhd^{-1}r_2}{d+cr_2\xi_2}\right)
    +\left(bhd^{-1}c^{-1}-\frac{bhd^{-1}r_2}{d\xi_2+cr_2}\right)\xi_2^{-1}\Bigg]d\xi_2\Bigg\}\\
  &=&bhd^{-1}c^{-1}\left[\widetilde{g}(z(0))-\widetilde{g}(z(-\frac{d}{cr_2}))\right]\\
  &\rightarrow& bhd^{-1}c^{-1}\left[\widetilde{g}(z(0))-\widetilde{g}(z(-\frac{d}{c}))\right]\\
  &=&2bhd^{-1}c^{-1}\log\frac{c\gamma}{c\gamma-d\eta}.
\end{eqnarray*}

\subsection*{Proof of (\ref{testE}) and (\ref{testVar})}
Because \textbf{$\xi$}  and  \textbf{$\eta$}   are Gaussian
variables, for real case, $\beta=E|$\textbf{$\xi$}$|^4-3=0,$ then
(\ref{E1betax}), (\ref{E1betay}) and (\ref{cov1betax})
are all 0. Consider (\ref{E1}) and (\ref{cov1}),
as $y_{n_k} \rightarrow y_k, ~k=1,2,$, by the computations done in the
proof of  Lemma~\ref{lem1},  we see that termes tending to zero 
could be neglected in the considered contour integrals.
Hence we can put $ y_{n_k}=y_k, k=1,2 $ and   use
$$
f(x)=\log(y_{1}+y_{2}x)-\frac{y_{2}}{y_{1} +y_{2}}\log
x-\log(y_{1}+y_{2})
$$
instead of
$f(x)=\displaystyle{\log(y_{n_1}+y_{n_2}x)-\frac{y_{n_2}}{y_{n_1}
    +y_{n_2}}\log x-\log(y_{n_1}+y_{n_2})}$.  
Consider  the variable change 
$x=(1-y_2)^{-2}(1+h^2-2h\cos \theta)$, where
$z(\xi)=(1-y_2)^{-2}\left[1+h^2+2h\mathcal{R}(\xi)\right], \quad h
=\sqrt{y_1+y_2-y_1y_2}$. As 
\begin{eqnarray*}
  \log(y_{n_1}+y_{n_2}z(\xi))& = & \log\left(\frac{\left|h+y_{2}\xi\right|^2}{(1-y_2)^2}\right),\\
  \log(z(\xi))& = &
  \log\left(\frac{\left|1+h\xi\right|^2}{(1-y_2)^2}\right),
\end{eqnarray*}
we have by  Lemma~\ref{lem1}, 
\begin{eqnarray*}
  m(f)&=&\displaystyle{\frac{1}{2}\left[\log\frac{(h^2-y_2^2)h^2}{(h^2-y_2^2)^2}
      -\frac{y_2}{y_1+y_2}\log\frac{(1-h^2)h^2}{(h-y_2h)^2}\right]}\\
  &=&\displaystyle{\frac{1}{2}\left[\log\left(\frac{y_1+y_2-y_1y_2}{y_1+y_2}\right)
      -\frac{y_1}{y_1+y_2}\log(1-y_2)-\frac{y_2}{y_1+y_2}\log(1-y_1)\right]},
\end{eqnarray*}
and 
\begin{eqnarray*}
  \upsilon(f)&=&\upsilon\big(\log(y_{n_1}+y_{n_2}x)\big)+\frac{y_2^2}{(y_1+y_2)^2}\upsilon
  \big(\log x\big)-\frac{2y_2}{y_1+y_2}\upsilon\big(\log x, \log
  (y_{n_1}+y_{n_2}x)\big)\\
  &=&2\log
  \frac{h^2}{h^2-y_2^2}+2\frac{y_2^2}{(y_1+y_2)^2}\log\frac{1}{1-h^2}-\frac{4y_2}{y_1+y_2}\log
  \frac{1}{1-y_2}\\
  &=&-\frac{2y_2^2}{(y_1+y_2)^2}\log(1-y_1)-\frac{2y_1^2}{(y_1+y_2)^2}\log(1-y_2)
  -2\log\frac{y_1+y_2}{y_1+y_2-y_1y_2}.
\end{eqnarray*}

\subsection*{Proof of $F_{y_{n_1}, y_{n_2}}(f)$}

By (\ref{f(x)}) and  the density of $F_{y_{n_1}, y_{n_2}}(f)$ (the
limiting distribution in (\ref{LSDden}) but with $y_{n_k}$ in place
of $y_k, k=1,2.$ ), where
$h_n=\sqrt{y_{n_1}+y_{n_2}-y_{n_1}y_{n_2}}$,
$a_n=(1-y_{n_2})^{-2}\left(1-\sqrt{y_{n_1}+y_{n_2}-y_{n_1}y_{n_2}}\right)^2$
and
$b_n=(1-y_{n_2})^{-2}\left(1+\sqrt{y_{n_1}+y_{n_2}-y_{n_1}y_{n_2}}\right)^2.$
Using the  substitution $
x=(1-y_{n_2})^{-2}\left(1+h_n^2-2h_n\cos\theta\right), \quad 0 <
\theta < \pi, $   we have 
\begin{eqnarray}
  \sqrt{(b_n-x)(x-a_n)}= \frac{2h_n\sin\theta}{(1-y_{n_2})^2},
  \quad&\quad
  dx=\displaystyle\frac{2h_n\sin\theta d\theta}{(1-y_{n_2})^2};\nonumber\\
  x=\frac{\left|1-h_ne^{i\theta}\right|^2}{(1-y_{n_2})^2},\quad&\quad
  \displaystyle{y_{n_1}+y_{n_2}x}=\displaystyle{\frac{\left|h_n-y_{n_2}e^{i\theta}\right|^2}
    {(1-y_{n_2})^2}}.\nonumber
\end{eqnarray}
Therefore,
\begin{eqnarray*}
  &&{F^{y_{n_1},      y_{n_2}}(f)}\\[5mm]
  &=&{\int_{a_n}^{b_n}f(x)\frac{(1-y_{n_2})\sqrt{(b_n-x)(x-a_n)}}{2\pi
      x(y_{n_1}+y_{n_2}x)}dx}\\[6mm]
  &=&{(1-y_{n_2})\int_{a_n}^{b_n}
    \left[\log\left(y_{n_1}+y_{n_2}x\right)
      -\frac{y_{n_2}}{y_{n_1}+y_{n_2}}\log
      x\right]\frac{\sqrt{(b_n-x)(x-a_n)}}{2\pi x(y_{n_1}+y_{n_2}x)}dx}\\
  &&\quad\quad-{\log\left(y_{n_1}+y_{n_2}\right)}\\[6mm]
  &=&{\frac{2(1-y_{n_2})}{\pi}\int_{0}^{\pi}\left[
      \log\frac{\left|h_n-y_{n_2}e^{i\theta}\right|^2}{(1-y_{n_2})^2}
      -\frac{y_{n_2}}{y_{n_1}+y_{n_2}}\log\frac{\left|1-h_ne^{i\theta}\right|^2}
      {(1-y_{n_2})^2}\right]}\\[6mm] &&\quad\quad
  {\cdot\frac{h_n^2\sin^2\theta}{\left|1-h_ne^{i\theta}\right|^2
      \left|h_n-y_{n_2}e^{i\theta}\right|^2}d\theta}-{
    \log\left(y_{n_1}+y_{n_2}\right)}\\[6mm]
  &=&{\frac{2(1-y_{n_2})}{\pi}\int_{0}^{\pi} \left[
      \log\left|h_n-y_{n_2}e^{i\theta}\right|^2
      -\frac{y_{n_2}}{y_{n_1}+y_{n_2}}\log\left|1-h_ne^{i\theta}\right|^2
      \right]} \\[4mm]
  &&\cdot{\frac{h_n^2\sin^2\theta}
    {\left|1-h_ne^{i\theta}\right|^2\left|h_n-y_{n_2}e^{i\theta}\right|^2}d\theta}
  -2\left(1-\frac{y_{n_2}}{y_{n_1}+y_{n_2}}\right)
  \log(1-y_{n_2})-\log{(y_{n_1}+y_{n_2})}\\
  &=&{\Re\Bigg\{\frac{2(1-y_{n_2})}{\pi}\int_{0}^{2\pi}
    \left[ \log(h_n-y_{n_2}e^{i\theta})
      -\frac{y_{n_2}}{y_{n_1}+y_{n_2}}\log(1-h_ne^{i\theta})\right]}\\
  &&
      {\frac{h_n^2\sin^2\theta}{\left|1-h_ne^{i\theta}\right|^2
          \left|h_n-y_{n_2}e^{i\theta}\right|^2}d\theta\Bigg\}}-\frac{2y_{n_1}}{y_{n_1}+y_{n_2}}
      \log(1-y_{n_2})-\log{(y_{n_1}+y_{n_2})}\\[5mm]
      &=&{\Re\Bigg\{\frac{-(1-y_{n_2})}{2\pi i}\oint_{|z|=1}
        \left[\log(h_n-y_{n_2}z)-\frac{y_{n_2}}{y_{n_1}+y_{n_2}}\log(1-h_nz)\right]}\\[5mm]
      &&\cdot{\frac{h_n^2(z-z^{-1})^2}{z\left|1-h_nz\right|^2
          \left|h_n-y_{n_2}z\right|^2}dz\Bigg\}}-{\frac{2y_{n_1}}{y_{n_1}+y_{n_2}}
        \log(1-y_{n_2})-\log{(y_{n_1}+y_{n_2})}}\\[5mm]
      &=&{\Re\Bigg\{\frac{y_{n_2}-1}{y_{n_2}}\cdot\frac{1}{2\pi
          i}\oint_{|z|=1}\left[
          \log(h_n-y_{n_2}z)-\frac{y_{n_2}}{y_{n_1}+y_{n_2}}\log(1-h_nz)\right]}\\[6mm]
      &&\cdot{\frac{(z^2-1)^2}{z(z-h_n)(z-\frac{1}{h_n})
          (z-\frac{y_{n_2}}{h_n})(z-\frac{h_n}{y_{n_2}})}dz}\Bigg\}
      -{\frac{2y_{n_1}}{y_{n_1}+y_{n_2}}
        \log(1-y_{n_2})-\log{(y_{n_1}+y_{n_2})}}.
\end{eqnarray*}
There are three poles inside the unit circle: 0, $h_n, y_{n_2}/h_n$.
Their corresponding residues are
\begin{eqnarray}
  R(0) &=& \frac{y_{n_2}-1}{y_{n_2}}\log (h_n) ,\nonumber\\
  R(h_n)&=& \frac{(h_n^2-1)}{(h_n^2-y_{n_2})}\left[\log (h_n)+\log
    (1-y_{n_2})-\frac{y_{n_2}}{y_{n_1}+y_{n_2}}\log (1-h^2_n)
    \right],\nonumber\\
  R(\frac{y_{n_2}}{h_n})&=&
  \frac{(y_{n_2}^2-h_n^2)}{y_{n_2}(y_{n_2}-h_n^2)} \left[\log
    (h^2_n-y_{n_2}^2)-\log (h_n)-\frac{y_{n_2}}{y_{n_1}+y_{n_2}}\log
    (1-y_{n_2}) \right].\nonumber
\end{eqnarray}
Therefore,
\begin{eqnarray*}
  F^{y_{n_1}, y_{n_2}}(f)&=&
  R(0)+R(h_n)+R(\frac{y_{n_2}}{h_n})\displaystyle{-\frac{2y_{n_1}}{y_{n_1}+y_{n_2}}
    \log(1-y_{n_2})-\log{(y_{n_1}+y_{n_2})}} \\
  &=&\displaystyle{\frac{-(y_{n_1}+y_{n_2}-y_{n_1}y_{n_2})}{y_{n_1}y_{n_2}}\log{(y_{n_1}+y_{n_2}-y_{n_1}y_{n_2})}}\\[2mm]
  &&\quad
  +\displaystyle{\frac{(y_{n_1}+y_{n_2}-y_{n_1}y_{n_2})}{y_{n_1}y_{n_2}}\log{(y_{n_1}+y_{n_2})}}
  +\displaystyle{\frac{y_{n_1}(1-y_{n_2})}{y_{n_2}(y_{n_1}+y_{n_2})}\log{(1-y_{n_2})}}\nonumber\\
  &&\quad\quad\quad+\displaystyle{\frac{y_{n_2}(1-y_{n_1})}{y_{n_1}(y_{n_1}+y_{n_2})}\log{(1-y_{n_1})}}.\nonumber
\end{eqnarray*}

\subsection*{Proof of (\ref{m1f}) and (\ref{v1f})}

Because \textbf{x} and \textbf{y}  are random variables from
normalized $t$-distribution with 5 degree of freedom, \textbf{x} and
\textbf{y} are $i.i.d.$, $E\textbf{x}=E\textbf{y}=0, $ $
E|\textbf{x}|^2=E|\textbf{y}|^2=1 $ and $
E|\textbf{x}|^4=E|\textbf{y}|^4=9. $   For real case,
$\beta=E|$\textbf{$\xi$}$|^4-3=6,$ (\ref{E1}) and (\ref{cov1}) items
are the same to the Gaussian variables. Consider the items
(\ref{E1betax}), (\ref{E1betay}) and (\ref{cov1betax}). As the same
explanation in Proof of (\ref{testE}) and (\ref{testVar}), we use $
f(x)=\log(y_{1}+y_{2}x)-\frac{y_{2}}{y_{1} +y_{2}}\log
x-\log(y_{1}+y_{2}) $ instead.

For (\ref{E1betax}), we have 
\begin{eqnarray*}
  &&\frac{\beta\cdot y_1(1-y_2)^2}{2\pi i \cdot h^2} \oint_{|\xi|=1}
  \left[\log\frac{|h+y_2\xi|^2}{(1-y_2)^2}-\frac{y_2}{y_1+y_2}\log\frac{|1+h\xi|^2}{(1-y_2)^2}-\log(y_1+y_2)\right]\\
  &&\quad\quad\quad\quad\quad\quad \quad\quad\quad\quad\quad\quad
  \quad\quad\quad\quad\quad\quad\quad\quad\quad\quad\quad\quad
  \quad\quad\quad\cdot \frac{1}{(\xi+\frac{y_2}{hr})^3} d\xi\\
  &=&\frac{\beta\cdot y_1(1-y_2)^2}{2\pi i \cdot h^2} \oint_{|\xi|=1}
  2
  \mathcal{R}\Big\{\log(h+y_2\xi)-\frac{y_2}{y_1+y_2}\log(1+h\xi)\Big\}\cdot
  \frac{1}{(\xi+\frac{y_2}{hr})^3} d\xi\\
  &=&\frac{\beta\cdot y_1(1-y_2)^2}{2\pi i \cdot h^2}
  \oint_{|\xi|=1}\Bigg\{\log(h+y_2\xi)+\log(h+y_2\overline{\xi})\\
  && \quad\quad\quad\quad\quad\quad\quad\quad\quad\quad
  -\frac{y_2}{y_1+y_2}\Big[\log(1+h\xi)+\log(1+h\overline{\xi})\Big]\Bigg\}
  \cdot\frac{1}{(\xi+\frac{y_2}{hr})^3} d\xi\\
  &=&\frac{\beta\cdot y_1(1-y_2)^2}{2\pi i \cdot h^2}
  \oint_{|\xi|=1}\left[\log(h+y_2\xi)-\frac{y_2}{y_1+y_2}\log(1+h\xi)\right]
  \\ &&\quad\quad\quad\quad\quad\quad\quad\quad\quad\quad\quad\quad
  \quad\quad\quad\quad\quad\quad\quad\quad\quad
  \left[\frac{1}{(\xi+\frac{y_2}{hr})^3}
    +\frac{\left(\frac{hr}{y_2}\right)^3\xi}{(\frac{hr}{y_2}+\xi)^3}
    \right] d\xi\\
  &=&\frac{\beta\cdot y_1(1-y_2)^2}{2\pi i \cdot h^2} \cdot 2\pi \cdot
  \frac{1}{2}
  \left[\log(h+y_2\xi)-\frac{y_2}{y_1+y_2}\log(1+h\xi)\right]''\Bigg|_{\xi=-\frac{y_2}{hr}}\\
  &=&\frac{\beta\cdot y_1(1-y_2)^2}{2 h^2}
  \left[-\frac{y_2^2}{(h+y_2\xi)^2}+\frac{y_2}{y_1+y_2}\frac{h^2}{(1+h\xi)^2}\right]\Bigg|_{\xi=-\frac{y_2}{hr}}\\
  &=&\frac{\beta y_1^2y_2}{2(y_1+y_2)^2}.
\end{eqnarray*}

For (\ref{E1betay}), we have 
\begin{eqnarray*}
  &&\frac{\beta\cdot (1-y_2)}{4\pi i} \oint_{|\xi|=1}
  \left[\log\frac{|h+y_2\xi|^2}{(1-y_2)^2}-\frac{y_2}{y_1+y_2}\log\frac{|1+h\xi|^2}{(1-y_2)^2}-\log(y_1+y_2)\right]\\
  &&\quad\quad\quad\quad\quad\quad\quad\quad\quad\quad\quad\quad\quad\cdot
  \frac{\xi^2-\frac{y_2}{h^2r^2}}{(\xi+\frac{y_2}{hr})^2}\left[
    \frac1{\xi-\frac{\sqrt{y_2}}{hr}}+\frac1{\xi+\frac{\sqrt{y_2}}{hr}}-\frac2{\xi+\frac{y_2}{hr}}\right]d\xi\\
  &=&\frac{\beta\cdot (1-y_2)y_2}{2\pi i \cdot h} \oint_{|\xi|=1}
  \left[\log\frac{|h+y_2\xi|^2}{(1-y_2)^2}-\frac{y_2}{y_1+y_2}\log\frac{|1+h\xi|^2}{(1-y_2)^2}-\log(y_1+y_2)\right]\\
  && \quad\quad\quad\quad\quad\quad\quad\quad\quad\quad\quad\quad\quad
  \quad\quad\quad\quad\quad\quad\quad\quad\quad\quad\quad\quad\cdot
  \left[\frac{\xi+\frac{1}{hr}}{(\xi+\frac{y_2}{hr})^3}\right]
  d\xi\\
  &=&\frac{\beta\cdot (1-y_2)y_2}{2\pi i \cdot h} \oint_{|\xi|=1}2
  \mathcal{R}\Big\{\log(h+y_2\xi)-\frac{y_2}{y_1+y_2}\log(1+h\xi)\Big\}
  \cdot \left[\frac{\xi+\frac{1}{hr}}{(\xi+\frac{y_2}{hr})^3}\right]
  d\xi\\
  &=&\frac{\beta\cdot (1-y_2)y_2}{2\pi i \cdot h}\oint_{|\xi|=1}\Big[
    \log(h+y_2\xi)+\log(h+y_2\overline{\xi})\\
    &&\quad\quad\quad\quad\quad\quad\quad\quad
    -\frac{y_2}{y_1+y_2}\left(\log(1+h\xi)+\log(1+h\overline{\xi})\right)
    \Big]\cdot
  \left[\frac{\xi+\frac{1}{hr}}{(\xi+\frac{y_2}{hr})^3}\right]d\xi\\
  &=&\frac{\beta\cdot (1-y_2)y_2}{2\pi i \cdot
    h}\oint_{|\xi|=1}\left[\log(h+y_2\xi)-\frac{y_2}{y_1+y_2}\log(1+h\xi)\right]\\
  &&\quad\quad\quad\quad\quad\quad\quad\cdot\Big[\frac{\xi+\frac{1}{hr}}{(\xi+\frac{y_2}{hr})^3}+
    \frac{\frac{h^2r^2}{y_2^3}(\xi+hr)}{(\xi+\frac{hr}{y_2})^3}\Big]d\xi\\
  &=&\frac{\beta\cdot (1-y_2)y_2}{2\pi i \cdot h}2\pi i \cdot
  \frac1{2}\left([\log(h+y_2\xi)-\frac{y_2}{y_1+y_2}\log(1+h\xi)]
  \cdot(\xi+\frac{1}{hr})\right)''\Bigg|_{\xi=-\frac{y_2}{hr}}\\
  &=&\frac{\beta y_1y_2^2}{2(y_1+y_2)^2}.
\end{eqnarray*}
Therefore, 
\begin{eqnarray}
  m_1(f)&=&
  \frac{1}{2}\Big[\log\left(\frac{y_1+y_2-y_1y_2}{y_1+y_2}\right)
    -\frac{y_1}{y_1+y_2}\log(1-y_2)-\frac{y_2}{y_1+y_2}\log(1-y_1)\nonumber\\
    \quad &+& \frac{6 y_1^2y_2}{(y_1+y_2)^2}+\frac{6
      y_1y_2^2}{(y_1+y_2)^2}\Big].\nonumber
\end{eqnarray}

For covariance, we have 
\begin{eqnarray*}
  &&\oint_{|\xi|=1}
  \frac{f\left(\frac{1+h^2+2h\mathcal{R}(\xi)}{(1-y_2)^2}\right)}{(\xi+\frac{y_2}{hr})^2}d\xi\\
  &=& \oint_{|\xi|=1}
  \left[\log\frac{|h+y_2\xi|^2}{(1-y_2)^2}-\frac{y_2}{y_1+y_2}\log\frac{|1+h\xi|^2}{(1-y_2)^2}-\log(y_1+y_2)\right]
  \cdot \frac{1}{(\xi+\frac{y_2}{hr})^2} d\xi\\
  &=& \oint_{|\xi|=1} 2
  \mathcal{R}\Big\{\log(h+y_2\xi)-\frac{y_2}{y_1+y_2}\log(1+h\xi)\Big\}\cdot
  \frac{1}{(\xi+\frac{y_2}{hr})^2} d\xi\\
  &=& \oint_{|\xi|=1}\Bigg\{\log(h+y_2\xi)+\log(h+y_2\overline{\xi})\\
  && \quad\quad
  -\frac{y_2}{y_1+y_2}\Big[\log(1+h\xi)+\log(1+h\overline{\xi})\Big]\Bigg\}
  \cdot\frac{1}{(\xi+\frac{y_2}{hr})^2} d\xi\\
  &=&
  \oint_{|\xi|=1}\left[\log(h+y_2\xi)-\frac{y_2}{y_1+y_2}\log(1+h\xi)\right]
  \left[\frac{1}{(\xi+\frac{y_2}{hr})^2}
    +\frac{\left(\frac{hr}{y_2}\right)^2}{(\xi+\frac{hr}{y_2})^2}
    \right] d\xi\\
  &=& 2\pi i \cdot \left[\log(h+y_2\xi)-\frac{y_2}{y_1+y_2}\log(1+h\xi)\right]'\Bigg|_{\xi=-\frac{y_2}{hr}}\\
  &=& \pi i \cdot
  \left[\frac{y_2}{h+y_2\xi}+\frac{y_2}{y_1+y_2}\frac{h}{1+h\xi}\right]\Bigg|_{\xi=-\frac{y_2}{hr}}\\
  &=&0.
\end{eqnarray*}

So, (\ref{cov1betax}) becomes,
\begin{eqnarray*}
  &&-\frac{\beta \cdot (y_1+y_2)(1-y_2)^2}{4\pi^2h^2}
  \oint_{|\xi_1|=1}
  \frac{f\left(\frac{1+h^2+2h\mathcal{R}(\xi_1)}{(1-y_2)^2}\right)}{(\xi_1+\frac{y_2}{hr_1})^2}d\xi_1
  \oint_{|\xi_2|=1}
  \frac{f\left(\frac{1+h^2+2h\mathcal{R}(\xi_2)}{(1-y_2)^2}\right)}{(\xi_2+\frac{y_2}{hr_2})^2}d\xi_2=0
\end{eqnarray*}
Finally, 
\begin{equation}
  \upsilon_1(f)=-\frac{2y_2^2}{(y_1+y_2)^2}\log(1-y_1)-\frac{2y_1^2}{(y_1+y_2)^2}\log(1-y_2)
  -2\log\frac{y_1+y_2}{y_1+y_2-y_1y_2}.\nonumber
\end{equation}

\end{document}